\documentclass[a4paper,11pt]{amsart}

\usepackage{amssymb}
\usepackage{epsfig}
\usepackage{amsfonts}
\usepackage{amsmath}
\usepackage{euscript}
\usepackage{amscd}
\usepackage{amsthm}
\usepackage{tikz-cd}
\usepackage{mathtools}
\DeclareMathAlphabet{\mathpzc}{OT1}{pzc}{m}{it}
\usepackage[pagebackref,hypertexnames=false, colorlinks, citecolor=red,linkcolor=blue, urlcolor=red]{hyperref}
\usepackage{mathrsfs}

\DeclareSymbolFont{SY}{U}{psy}{m}{n}
\DeclareMathSymbol{\emptyset}{\mathord}{SY}{'306}

\theoremstyle{plain}

\newtheorem{thm}{Theorem}[section]
\newtheorem{cor}[thm]{Corollary}
\newtheorem{lem}[thm]{Lemma}
\newtheorem{prop}[thm]{Proposition}
\theoremstyle{definition}
\newtheorem{defn}[thm]{Definition}
\newtheorem{rem}[thm]{Remark}

\numberwithin{equation}{section}
\def\A{{\mathcal A}}
\def\C{{\mathbb C}}
\def\D{{\mathbb D}}

\def\O{{\mathcal O}}
\def\H{{\mathcal H}}
\def\HM{{\mathscr M}}

\def\B{{\mathcal B}}
\def\U{{\mathcal U}}
\def\dbar{\bar\partial}
\def\norm#1{\left\|{#1}\right\|}
\def\N{\mathbb{N}}
\def\R{\mathbb{R}}

\def\J{\mathcal J}

\def\s{\psi}
\def\si{\sigma}

\def\l{\lambda}

\def\O{\Omega}
\def\f{\textbf{\textit{s}}}

\def\ra{\rightarrow}
\def\ov{\overline}

\def\a{\alpha}
\def\b{\beta}
\def\g{\gamma}
\def\G{\Gamma}
\def\e{\varepsilon}

\def\d{\delta}

\def\h{ hermitian holomorphic vector bundle}
\def\hl{ hermitian holomorphic line bundle}

\def\w{with respect to }
\def\del{\partial}
\def\T{\boldsymbol{T}}
\def\M{\boldsymbol{M}}
\def\rkhs{reproducing kernel Hilbert space }
\def\<{\langle}
\def\>{\rangle}
\def\Z{\mathcal{Z}}

\def\z{\zeta}
\def\bz{\boldsymbol{z}}
\def\bw{\boldsymbol{w}}
\def\aut{\text{Aut}}

\newcommand{\beq}{\begin{eqnarray}}
\newcommand{\eeq}{\end{eqnarray}}
\newcommand{\bea}{\begin{eqnarray}}
\newcommand{\eea}{\end{eqnarray}}
\newcommand{\Bea}{\begin{eqnarray*}}
\newcommand{\Eea}{\end{eqnarray*}}

\newcounter{cnt1}
\newcounter{cnt2}
\newcounter{cnt3}
\newcommand{\blr}{\begin{list}{$($\roman{cnt1}$)$}
 {\usecounter{cnt1} \setlength{\topsep}{0pt}
 \setlength{\itemsep}{0pt}}}
\newcommand{\bla}{\begin{list}{$($\alph{cnt2}$)$}
 {\usecounter{cnt2} \setlength{\topsep}{0pt}
 \setlength{\itemsep}{0pt}}}
\newcommand{\bln}{\begin{list}{$($\arabic{cnt3}$)$}
 {\usecounter{cnt3} \setlength{\topsep}{0pt}
 \setlength{\itemsep}{0pt}}}
\newcommand{\el}{\end{list}}

\begin{document}

%%%%%%%%%%%%%%%%%%%%%%%%%%%%%%%%%%%%%%%%%%%%%%%%%%%%%%%%%%%%%%%%%%%%%%%%%%%%%%%%%%%
\title [HOMOGENEOUS OPERATORS OBTAINED FROM QUOTIENT MODULES]{ON IRREDUCIBILITY OF A CERTAIN CLASS OF HOMOGENEOUS OPERATORS OBTAINED FROM QUOTIENT MODULES}
%%%%%%%%%%%%%%%%%%%%%%%%%%%%%%%%%%%%%%%%%%%%%%%%%%%%%%%%%%%%%%%%%%%%%%%%%%%%%%%%%%%
\author{SHIBANANDA BISWAS, PRAHLLAD DEB AND SUBRATA SHYAM ROY}
%\address{Department of Mathematics and Statistics, Indian Institute of Science Education and Research Kolkata, Mohanpur - 741246, West Bengal, India}
%\email{prahllad.deb@gmail.com}

\thanks{This work is supported by Senior Research Fellowship funded by IISER Kolkata.}


%%%%%%%%%%%%%%%%%%%%%%%%%%%%%%%%%%%%%%%%%%%%%%%%%%%%%

\begin{abstract}
Let $\O \subset \C^m$ be an open, connected and bounded set and $\A(\O)$ be a function algebra of holomorphic functions on $\O$. Suppose that $\HM_q$ is the quotient Hilbert module obtained from a submodule of functions in a Hilbert module $\HM$ vanishing to order $k$ along a smooth irreducible complex analytic set $\Z\subset\O$ of codimension at least $2$. In this article, we prove that the compression of the multiplication operators onto $\HM_q$ is homogeneous \w a suitable subgroup of the automorphism group $\aut(\O)$ of $\O$ depending upon a subgroup $G$ of $\aut(\O)$ whenever the tuple of multiplication operators on $\HM$ is homogeneous \w $G$ and both $\HM$ as well as $\HM_q$ are in the Cowen-Douglas class. We show that these compression of multiplication operators might be reducible even if the tuple of multiplication operators on $\HM$ is irreducible by exhibiting a concrete example. Moreover, the irreducible components of these reducible operators are identified as \textit{Generalized Wilkins'} operators. 
\end{abstract}

\maketitle
%%%%%%%%%%%%%%%%%%%%%%%%%%%%%%%%%%%%%%%%%%%%%%%%%%%%%%%%%%%%%%%%%%%%%%%%%%%%%%%%%%%
\section{Introduction}\label{intro}

Let $\H$ be a complex separable Hilbert space, $\O \subset \C^m$ be a bounded domain and $G$ be a subgroup of the group $\aut(\O)$ of biholomorphic automorphisms of $\O$ acting on it transitively. Then an $m$-tuple of commuting bounded linear operators $\boldsymbol{T}=(T_1,\hdots,T_m)$ is said to be $G$-homogeneous if the Taylor joint spectrum of $\boldsymbol{T}$ is contained in $\ov{\O}$ and, for every $\boldsymbol{g}=(g_1,\hdots,g_m) \in G$ there exists a unitary operator $U_{\boldsymbol{g}}$ such that $g_j(\boldsymbol{T}) = U_{\boldsymbol{g}}^*T_jU_{\boldsymbol{g}}$, $1 \leq j \leq m$.

For $\H$ a {\rkhs} $\H_K$ with an $r \times r$ matrix valued reproducing kernel $K$ on  $\O$, there is a natural action of $G$ on the space of multiplication operators on $\H_K$, namely,
 $$\boldsymbol{g}\cdot\M_{f} := \M_{f \circ\boldsymbol{g}},$$
 for all $\boldsymbol{g}\in G$ and holomorphic functions $f$ on $\O$ such that the multiplication operator $\M_f$ is a bounded linear operator on $\H_K$. Then the $m$-tuple of multiplication operators $\M=(M_{z_1},\hdots,M_{z_m})$ by co-ordinate functions on $\H_K$ is said to be $G$-homogeneous if $\boldsymbol{g}\cdot\textbf{M} = \textbf{M}_{\textbf{g}}$ is unitarily equivalent to $\boldsymbol{M}$, for all $\boldsymbol{g} \in G$.

A prototypical family of examples of homogeneous operators is the multiplication operators by the co-ordinate function on the weighted Bergman space $\H^{(\l)}(\D)$ which are reproducing kernel Hilbert spaces of holomorphic functions on the open unit disc $\D$ with the reproducing kernel $K^{(\l)}(z,w)=(1-z\ov{w})^{-\l}$, $\l>0$. In fact, these are the only irreducible homogeneous operators in the Cowen-Douglas class $\mathrm B_1(\D)$ of rank $1$ over $\D$ as shown in  \cite{ACHOCD, CABSO}. Later, Wilkins in \cite{HVBCD} characterised all irreducible homogeneous operators in the Cowen-Douglas class $\mathrm B_2(\D)$ of rank $2$ over $\D$ by providing an explicit construction of such operators. According to Wilkins' construction, all irreducible homogeneous operators in $\mathrm B_2(\D)$ can be viewed as the compression of the multiplication operator $M^{(1)}\otimes I$ on $\H:=\H_{K_1}\otimes\H_{K_2}$ onto the quotient Hilbert space $\H_q$ obtained from the subspace consisting of functions in $\H$ vanishing of order $2$ along the diagonal subset $\{(z_1,z_2)\in\D^2:z_1=z_2\}$ of the open unit bi-disc $\D^2$ for some reproducing kernels $K_i$ on $\D$, $i=1,2$, such that the multiplication operators $M^{(i)}$, $i=1,2$, are homogeneous operators in $\mathrm B_1(\D)$. So in general, it is natural to ask if the compression of the tuple of multiplication operators $\M=(M_{z_1},\hdots,M_{z_m})$ by the co-ordinate functions on a reproducing kernel Hilbert module $\HM$ over $\A(\O)$ onto the quotient modules $\HM_q$ obtained from the submodule $\HM_0$ of functions in $\HM$ vanishing along a complex submanifold $\mathsf{Z}$ (see (i) in \ref{stdassumption} for definition) of higher order is homogeneous \w some group depending upon a subgroup $G$ of $\aut(\O)$, provided $\M$ is $G$-homogeneous. One of the main goals of this article is to study this question assuming that the pair of Hilbert modules $(\HM,\HM_q)$ is in $\mathrm B_{r,k}(\O,\mathsf Z)$ (cf. Definition \ref{CDq}). It is then seen that the compression of $\M$ onto $\HM_q$ in general is not irreducible (that is, there is no non-trivial reducing subspace) even if $\M$ is, by constructing an explicit examples of such operators.

Since we are interested in Hilbert modules in the Cowen-Douglas class, let us recall following \cite{OPOSE, GBKCD} that the Cowen-Douglas class $\mathrm B_r(\O)$ over $\O$ of rank $r$ consists of $m$-tuples $\T=(T_1,\hdots,T_n)$ of commuting bounded linear operators on a Hilbert space $\H$ such that every $z\in\O$ is a joint eigenvalue of $\T$ with $r$ dimensional joint eigenspace, the linear span of the eigenspaces is dense in $\H$ and the range of the operator $D_{\T-zI}:=\left(T_1 -z_1I,\hdots,T_n -z_nI\right)$ is closed in $\H\oplus\cdots\oplus\H$. It has been proved in \cite{GBKCD} that the corresponding $n$ - tuple of operators $\T$ is simultaneously unitarily equivalent to the adjoint of the $n$ - tuple of multiplication operators $\M=(M_{z_1},\hdots,M_{z_n})$ by the co-ordinate functions on a Hilbert space $\H_K$ of holomorphic functions on $\O^*:=\{\ov{z}:z\in\O\}$ possessing a reproducing kernel $K$. For a reproducing kernel Hilbert module $\HM$ of holomorphic functions on $\O^*$ over $\A(\O^*)$ (cf. Definition \ref{Hilbertmodule}), slightly abusing the terminology we say $\HM \in \mathrm B_r(\O)$ whenever the adjoint of $\M$ on $\HM$ is in $\mathrm B_r(\O)$. This class of Hilbert modules has been a bridge between the study of operator theory and hermitian holomorphic vector bundles owing to the observation that a reproducing kernel Hilbert module $\HM$ with the reproducing kernel $K$ in $\mathrm B_r(\O)$ possesses a holomorphic frame, namely, for each $1\leq j\leq r$, a holomorphic map $s_j:\O \ra \HM$ defined by $s_j(w)=K(\cdot,\ov{w})\sigma_j$ for the hermitian holomorphic vector bundle $E_{\HM} \ra \O$ where $E_{\HM}=\{(f,w): f \in \ker D_{\M^*-\ov{w}I} \}\subset \HM \times \O$ and $\{\sigma_1,\hdots,\sigma_r\}$ is the standard ordered basis for $\C^r$. Then $K(\ov{w},\ov{w})$ turns out to be the hermitian metric on the fibre $E_{\HM}|_{w}$ of $E_{\HM}$ over $w\in\O$. Moreover, homogeneity of the $m$-tuple of operators $\M$ on $\HM \in \mathrm B_r(\O)$ is same as the homogeneity of the bundle $E_{\HM}$ (cf. Definition \ref{defhom}) as shown in Theorem \ref{homeqgc}.

Following \cite{OUIOQM} note that each pair of Hilbert modules $(\HM,\HM_q)$ over $\A(\O^*)$ in $\mathrm B_{r,k}(\O,\mathsf Z)$ gives rise to the jet bundle $J^{(k)}E_{\HM} \ra \O$ of order $k$ of the vector bundle $E_{\HM} \ra \O$ relative to $\mathsf Z$ as described in Section \ref{hom}. Also, it turns out that two such quotient modules are unitarily equivalent if and only if the corresponding jet bundles restricted to the submanifold $\mathsf Z$ are isomorphic. So in view of Theorem \ref{homeqgc}, the homogeneity of the quotient module $\HM_q$ can be seen as that of the jet bundle $J^{(k)}E_{\HM} \ra \O$ restricted to $\mathsf Z$. This observation leads to Theorem \ref{jethom} showing that homogeneity of a holomorphic vector bundle $E \ra \O$ under the action of some subgroup $G$ of $\aut(\O)$ fixing the submanifold $\mathsf Z$ yields the homogeneity of the jet bundle $J^{(k)}E_{\HM} \ra \O$ restricted to $\mathsf Z$ with respect to the group $G_{\mathsf Z} = \{\phi \in \aut(\mathsf Z):\phi=\Phi|_{\mathsf{Z}}, \Phi \in G\}$ when $\mathsf Z$ is assumed to be the intersection $Z$ of a co-ordinate plane with $\O$. An explicit formula for the automorphism of the jet bundle corresponding to a group element $\phi \in G_{Z}$ is obtained in terms of that of the bundle $E \ra \O$. What is more, it is shown in Theorem \ref{thmhom0} that this automorphisms of the jet bundle turns out to be isometric bundle maps when the bundle $E \ra \O$ and the jet bundle $J^{(k)}E_{\HM} \ra \O$ are associated to a pair of Hilbert modules $(\HM,\HM_q)$ in $\mathrm B_{r,k}(\O,Z)$. Then with the help of the observation made in Proposition \ref{eqhom} that the homogeneity of the compression of the tuple of multiplication operators onto $\HM_q$ is independent of the change of co-ordinates, we extend Theorem \ref{thmhom0} to the case where the pair of Hilbert modules $(\HM,\HM_q)$ is assumed to be in $\mathrm B_{r,k}(\O,\mathsf Z)$ for $\mathsf Z$ as in part (i) of \ref{stdassumption}. This is Theorem \ref{thmhom}. Finally, we conclude Section \ref{hom} with an application of these results showing that the compression of the $m$-tuple of multiplication operators by the co-ordinate functions on the weighted Bergman space over unit poly-disc $\D^m$ onto the quotient space obtained by the subspace of functions vanishing of order $k$ along the submanifold $\{\boldsymbol{z} \in \D^m:z_1=\cdots=z_d\}$ are homogeneous with respect to the group M\"ob$(\D)^{m-d+1}$.

 In Section \ref{qmreducibility}, we discuss on irreducibility of the homogeneous operators obtained in Section \ref{hom}. We point out from the example given in \cite[Theorem 6.1]{OICHO} that the compression of the tuple of multiplication operators $\M$ onto $\H_q$, in general, is not irreducible. In this regard, we take $\O$ to be $\D^m$ and consider the subspace $\HM^{(n)}$ of weighted Bergman space $\H^{(\a)}(\D^m)$ with $\a=(\a_1,\hdots,\a_m) \in \R^{m}_{>0}$ over $\D^m$ consisting of functions vanishing of order $n$ along the submanifold $\Delta:=\{(z_1,\hdots, z_m) \in \D^m : z_1 =\cdots= z_m\}$. It follows that the pair of Hilbert modules $(\H^{(\a)}(\D^m),\HM_q^{(n)})$ is in $\mathrm B_{1,n}(\D^m,\Delta)$ (\cite[Theorem 4.10]{OUIOQM}) where $\HM^{(n)}_q=\H^{(\a)}(\D^m) \ominus \HM^{(n)}$. Consequently, this pair of Hilbert modules gives rise to a jet bundle $J^{(n)}E \ra \D^m$ of the hermitian holomorphic line bundle $E \ra \D^m$ associated to $\H^{(\a)}(\D^m)$. In one of the main results of Section \ref{qmreducibility}, we decompose these jet bundles into orthogonal subbundles while $m=3$. This is Theorem \ref{mthm2}. We then extend this result to the general $m$ in Theorem \ref{mthm3} with the help of which we prove that the compression of the multiplication operator $M_{z_1}$ onto $\HM_q^{(n)}$ is reducible (cf. Theorem \ref{reducibility}). Also, the irreducible factors are obtained and identified as the \textit{Generalized Wilkins' operatots} (\cite[Page 428]{HOPRMGS}).

\subsection{Standing assumption:}\label{stdassumption}
\begin{itemize}

\item[(i)] Throughout this article, we make a standing assumption that the connected complex submanifolds $\Z\subset \O$ is biholomorphic to some co-ordinate plane in $\C^m$. In other words, $\Z$ possesses a global admissible chart $(\O,\psi)$ by which we mean that $\psi:\O\ra\C^m$ is a biholomorphism onto its range such that it takes the form $\psi(z)=(\psi_1(z),\hdots,\psi_d(z),z_{d+1},\hdots,z_m)$ with $\psi(p)=0$ and $$\psi(\O\cap\Z)=\{\l=(\l_1,...,\l_m) \in \psi(\O):\l_1=\cdots=\l_d=0\}.$$ We denote such submanifolds by the letter $\mathsf{Z}$ to make difference from the general one $\Z$. So throughout this section we only consider connected complex submanifolds $\mathsf{Z} \subset \O$.\\

\item[(ii)] Let $\HM$ be a reproducing kernel Hilbert space with reproducing kernel $K$ on $\O$ and $\HM_0$ be the submodule of $\HM$ consisting of functions vanishing of order $k$ along the submanifold $\mathsf{Z}$. Then by the compression of the tuple multiplication operators $\M=(M_{z_1},\hdots,M_{z_n})$ on $\HM$ onto the quotient space $\HM_q:=\HM\ominus\HM_0$ we mean that the compression of $(M_{z_{d+1}},\hdots,M_{z_n})$ onto $\HM_q$. 
\end{itemize}

\section{Jet construction}\label{JC}

We begin this this section by recalling the definition of Hilbert module. 
\begin{defn}\label{Hilbertmodule}
Let $\O$ be a bounded domain in $\C^m$ and $\A(\O)$ be the unital Banach algebra obtained as the norm closure \w the supremum norm on $\ov{\O}$ of all functions holomorphic on a neighbourhood of $\ov{\O}$. A complex Hilbert space $\H$ is said to be a Hilbert module over $\A(\O)$ with module map $\A(\O)\times\H \overset{\pi}\ra \H$ by point wise multiplication such that the module action $\A(\O)\times \H \overset{\pi}\ra \H$ is norm continuous. When a Hilbert module $\HM$ possesses a reproducing kernel on $\O$, we say that $\HM$ is a reproducing kernel Hilbert module.
\end{defn}
Let $\HM$ be a Hilbert module over $\A(\O)$ and $\HM_0 \subset \HM$ be the submodule of $\C^r$-valued holomorphic functions on $\O$ vanishing along a connected complex submanifold $\mathsf{Z} \subset \O$ of codimension $d$, $d \geq 1$. Suppose that the tuple of multiplication operators $\M=(M_{z_1},\hdots,M_{z_m})$ by co-ordinate functions on $\HM$ is in $B_r(\O)$. We then consider the quotient module $$\HM_q:= \HM/\HM_0= \HM \ominus \HM_0,$$ in other words, we have the following exact sequence \beq\label{exact} 0 \ra \HM_0 \overset{i}{\ra} \HM \overset{P}{\ra} \HM_q \ra 0\eeq where $i$ is the inclusion map and $\pi$ is the quotient map. For $f \in \A(\O)$ and $h \in \HM$, we define the module action on the quotient module $\HM_q$ as
\beq fP(h)=P(fh)\eeq where we mean $(fh_1,\hdots,fh_r)$ by $fh$.

%\smallskip

With a proper shrinking of the domain and a suitable co-ordinate changing, we may further assume without loss of generality that $\mathsf{Z} = \{z=(z_1,\hdots,z_m)\in \O : z_1 = \cdots = z_d = 0\}$ and 
$$\HM_0 = \{h \in \HM : \del^{\a_1}_1\cdots\del^{\a_d}_d(h_j)|_{\mathsf{Z}}=0, 0 \leq |\a| \leq k-1, 1 \leq j \leq r\}.$$

%\smallskip

We refer reader to Section $3$ in \cite{OUIOQM} for validity of such a reduction in detail. Now in order to describe the jet construction relative to the submanifold $\mathsf{Z}$ following \cite{OUIOQM}, let $A= \{\a=(\a_1,\hdots,\a_m)\in (\N\cup\{0\})^d:|\a|<k\}$, $\{\e_{\a}\}_{\a\in A}$, $N$ be the cardinality of $A$  and $\{\s_j\}_{j=1}^r$ be the standard ordered bases for $\C^{N}$ and  $\C^r$, respectively. For $h \in \HM$, let us define 
\beq\label{defjv}\textbf{h}:=\sum_{i=1}^r\left(\sum_{\a\in A}\del^{\a} h_i \otimes \e_l\right)\otimes \sigma_i\eeq and consider the space
\beq\label{jcon} J(\H):=\{\textbf{h}:h \in \H\} \subset \H\otimes \C^{Nr}\eeq where $\del^{\a}=\del_1^{\a_1}\cdots\del_d^{\a_d}$ with $\del_i=\frac{\partial}{\partial z_i}$. Note that the mapping $J:\H \ra J(\H)$ defined by $h \mapsto \textbf{h} $ is injective and consequently, we can define an inner product on $J(\H)$ as follows $$\<J(h_1),J(h_2)\>_{J(\H)}:=\<h_1,h_2\>_{\H}$$ making $J$ to be an unitary transformation.

Recall from \cite[Proposition 4.1]{OUIOQM} whenever $\HM$ is a reproducing kernel Hilbert space with a reproducing kernel $K$, so is the Hilbert space $J(\HM)$ of holomorphic functions on $\O$ taking values in $\C^{Nr}$ with the reproducing kernel 
\beq\label{repk} (JK)^{kl}_{ij}(z,w)=\del^k\dbar^l K_{ij}(z,w)\,\,\,\text{for}\,\,\,0\leq l,k\leq N,\,\,1\leq i,j \leq r~\text{and}~z,w\in\O.\eeq 

We now define the module action of the ring $\A(\O)$ on $J(\HM)$ making it to be a Hibert module over $\A(\O)$ as follows. For $f \in \A(\O)$ and $\textbf{h} \in J(\HM)$, the module action $J_f : J(\HM) \ra J(\HM)$ is defined by $J_f(\textbf{h}):=\J(f)\cdot\textbf{h}$ where $\J(f)$ is an $N\times N$ complex matrix with entries \beq\label{modac}\mathcal{J}(f)_{lj}:={\a_1 \choose \b_1}\cdots{\a_d \choose \b_d}\del^{\a-\b}f\eeq where $\a,\b\in A$ and $\textbf{h}$ can be thought of an $N\times r$ matrix with $\textbf{h}_i:= \sum_{\a\in A}\del^{\a} h_i \otimes \e_{\a}$, $1 \leq i \leq r$, as column vectors. It turns out that $\J(f)$ is a lower triangular matrix of the form
\[ \J(f)=
    \left(
    \begin{array}{ccccc}
     f                                    \\
      & \ddots           &   & \text{\huge0}\\
    \vdots  &    \mathcal{J}(f)_{lj}           & \ddots               \\
      &  &   &            \\
     \del^{N} f &  \hdots    &  \hdots &   &  f
    \end{array}
    \right).
\]
Note that $J$ is a module isomorphism between $\HM$ and $J(\HM)$ as it is clear from a simple calculation which is essentially an application of Leibniz rule. For detail calculations we refer reader to \cite[Section 4]{OUIOQM}.

In this set up, it is well known \cite[Theorem 4.5]{OUIOQM} that the quotient module $J(\HM)_q$ is unitarily equivalent to the module $J(\HM)|_{\text{res}\mathsf{Z}}$ over $\A(\O)|_{\text{res}\mathsf{Z}}$. It was also pointed out in Theorem 4.10 in \cite{OUIOQM} that the compression of the tuple of multiplication operators $\M$ on $\HM$ onto the quotient module $\HM_q$  lies in $\mathrm B_{N}(\mathsf{Z})$ whenever $\HM \in \mathrm B_1(\O)$, provided the reproducing kernel of $\HM$ has diagonal power series expansion where $\HM_q$ is restricted to a module over $\A(\mathsf{Z})$. So this phenomenon leads us to consider the following definition \cite[Definition 4.11]{OUIOQM}.

\begin{defn}\label{CDq} 
Let $\O \subset \C^m$ be bounded domain and $\mathsf{Z} \subset \O$ be the connected complex submanifold $\mathsf{Z}$ of codimension $d$. Then we say that the pair of Hilbert modules $(\HM,\HM_q)$ over the algebra $\A(\O^*)$ (here, by $\O^*$ we mean the set $\O^*=\{ \ov{w} : w \in \O \}$) is in $\mathrm B_{r,k}(\O,\mathsf{Z})$ if
\begin{enumerate}
\item there exists a resolution of the module $\HM_q$ as in $\eqref{exact}$ where the module $\HM$ over the algebra $\A(\O^*)$ appearing in the resolution is in $\mathrm B_r(\O)$;
\item for $f \in \A(\O^*)$, the restriction of the map $J_f$ to the submanifold $\mathsf Z^*$ defines the module action on $J(\HM)|_{\text{res}\mathsf{Z^*}}$ which is an isomorphic copy of $\HM_q$; and
\item the quotient module $\HM_q$ as a module over $\A(\O^*)|_{\mathsf{Z}^*}$ is in $\mathrm B_{Nr}(\mathsf{Z})$ where $N$ is the cardinality of the set $A= \{\a=(\a_1,\hdots,\a_m)\in (\N\cup\{0\})^d:|\a|<k\}$.
\end{enumerate}
\end{defn} 

\begin{rem} We now illustrate the jet construction described above with an example. For $\a,\b,\g \geq 0$, let $\H^{(\a,\b,\g)}$ be the reproducing kernel Hilbert space on $\D^3$ with the reproducing kernel 
\beq K^{(\a,\b,\g)}(\bz,\bw):=(1-z_1\ov{w}_1)^{-\a}(1-z_2\ov{w}_2)^{-\b}(1-z_3\ov{w}_3)^{-\g},\eeq
for $\bz,\bw \in \D^3$. Furthermore, the natural action of $\C[\bz]$ on $\H^{(\a,\b,\g)}$ makes it a Hilbert module over $\C[\bz]$ and for $\a,\b,\g > 0$, $\H^{(\a,\b,\g)}$ becomes a module over $\A(\D^3)$. 

Let us now consider the submodule $\H^{(\a,\b,\g)}_0$ consisting of holomorphic functions in $\H^{(\a,\b,\g)}$ which vanish upto order $2$ along the diagonal $\Delta:=\{\textbf{z}=(z_1,z_2,z_3) \in \D^3: z_1=z_2=z_3\}$, that is, $$\H^{(\a,\b,\g)}_0 =\{f \in \H^{(\a,\b,\g)}: f = \del_1 f =\del_2 f = \del_3 f =0 \text{ on } \Delta\}.$$
Let $\H^{(\a,\b,\g)}_q=\H^{(\a,\b,\g)}\ominus\H_0^{(\a,\b,\g)}$ be the quotient module. Since the reproducing kernel $K^{(\a,\b,\g)}$ possesses a diagonal power series, it follows from \cite[Theorem 4.10]{OUIOQM} that $(\H^{(\a,\b,\g)},\H^{(\a,\b,\g)}_q)$ is in $\mathrm B_{1,2}(\D^3,\Delta)$.

%\smallskip

We now compute the reproducing kernel of the quotient module $\H^{(\a,\b,\g)}_q$ by exhibiting an orthonormal basis of it. So to begin with, observe that the submodule $\H^{(\a,\b,\g)}_0$ is the closure of the ideal $I:=<(z_1-z_2)^2,(z_1-z_2)(z_1-z_3),(z_1-z_3)^2>$ in the Hilbert space $\H^{(\a,\b,\g)}$ verifying that $B:=\{z_1^i z_2^j z_3^k(z_1-z_2)^2,z_1^i z_2^j z_3^k(z_1-z_2)(z_1-z_3),z_1^i z_2^j z_3^k(z_1-z_3)^2: i,j,k \in \N \cup \{0\}\}$ is a spanning set for the submodule $\H^{(\a,\b,\g)}_0$. Therefore, it is enough to find an orthonormal basis for the orthogonal complement of $B$. An easy but tedious calculation shows that $\{e^{(p)}_1,e^{(p)}_2,e^{(p)}_3:p \in \N \cup \{0\}\}$ forms a basis for the quotient module $\H^{(\a,\b,\g)}_q $ where
$$
e^{(p)}_1 \mapsto
\left(
\begin{array}{c}
{-(\a+\b+\g)\choose p}^{\frac{1}{2}}z_1^p\\
\a\sqrt{\frac{p}{\a+\b+\g}}{-(\a+\b+\g+1)\choose (p-1)}^{\frac{1}{2}}z^{p-1}_1\\
\b\sqrt{\frac{p}{\a+\b+\g}}{-(\a+\b+\g+1)\choose (p-1)}^{\frac{1}{2}}z^{p-1}_1\\
\end{array}
\right),
~e^{(p)}_2  \mapsto
\left(
\begin{array}{c}
0\\
\frac{\a\b}{\sqrt{\b(\a+\g)}}\frac{1}{\sqrt{\a+\b+\g}}{-(\a+\b+\g+2)\choose (p-1)}^{\frac{1}{2}}z^{p-1}_1\\
\frac{\b\g}{\sqrt{\b(\a+\g)}}\frac{1}{\sqrt{\a+\b+\g}}{-(\a+\b+\g+2)\choose (p-1)}^{\frac{1}{2}}z^{p-1}_1\\
\end{array}
\right)$$
$$\text{and}~~
e^{(p)}_3 \mapsto
\left(
\begin{array}{c}
0\\
\sqrt{\frac{\a\g}{\a+\g}}{-(\a+\b+\g+2)\choose (p-1)}^{\frac{1}{2}}z^{p-1}_1\\
-\sqrt{\frac{\a\g}{\a+\g}}{-(\a+\b+\g+2)\choose (p-1)}^{\frac{1}{2}}z^{p-1}_1\\
\end{array}
\right).\\
$$

%\smallskip

This allows us to compute the reproducing kernel of the quotient module $\H^{(\a,\b,\g)}_q$ as follows

$$K_q(\bz,\bw)=\sum_{p=0}^{\infty}e^{(p)}_1(\bz)\cdot e^{(p)}_1(\bw)^* + e^{(p)}_2(\bz)\cdot e^{(p)}_2(\bw)^* + e^{(p)}_3(\bz)\cdot e^{(p)}_3(\bw)^*,\,\,\bz,\bw\in \Delta$$
which is a $3\times 3$ matrix valued function $((K_q(\bz,\bz)_{ij}))_{i,j=1}^3$ on $\Delta$ as expected. To compute the kernel $K_q(\bz,\bz)$ for $\bz \in \Delta$ we note, for $\bz=(z_1,z_1,z_1)$ in $\Delta$, that

$$
K_q(\bz,\bz)_{11} = K^{(\a,\b,\g)}(\bz,\bz),~
K_q(\bz,\bz)_{12} = \del_1K^{(\a,\b,\g)}(\bz,\bz),~
K_q(\bz,\bz)_{13} = \del_2K^{(\a,\b,\g)}(\bz,\bz),
$$
\Bea
K_q(\bz,\bz)_{23} &=& \frac{\a\b}{\a+\b+\g}\frac{d}{d|z_1|^2}\left({|z_1|^2(1-|z_1|^2)^{-(\a+\b+\g+1)}}\right)\\ &+& \left(\frac{\a\b\g}{(\a+\g)(\a+\b+\g)}-\frac{\a\g}{\a+\g}\right)(1-|z_1|^2)^{-(\a+\b+\g+2)}\\
&=& \a\b |z_1|^2 (1-|z_1|^2)^{-(\a+\b+\g+2)}\\ &=& \dbar_1\del_2K^{(\a,\b,\g)}(\bz,\bz),\\
K_q(\bz,\bz)_{22} &=& \frac{{\a}^2}{\a+\b+\g}\frac{d}{d|z_1|^2}\left({|z_1|^2(1-|z_1|^2)^{-(\a+\b+\g+1)}}\right)\\ &+& \left(\frac{{\a}^2\b}{(\a+\g)(\a+\b+\g)}+\frac{\a\g}{\a+\g}\right)(1-|z_1|^2)^{-(\a+\b+\g+2)}\\
&=& [\a+{\a}^2 |z_1|^2] (1-|z_1|^2)^{-(\a+\b+\g+2)}\\ &=& \del_1\dbar_1K^{(\a,\b,\g)}(\bz,\bz),\\
\Eea
and similar calculations also yield that $K_q(\bz,\bz)_{21}=\dbar_1K^{(\a,\b,\g)}(\bz,\bz)$, $K_q(\bz,\bz)_{31}=\dbar_2 K^{(\a,\b,\g)}(\bz,\bz)$, $K_q(\bz,\bz)_{32}=\dbar_2\del_1K^{(\a,\b,\g)}(\bz,\bz)$, and $K_q(\bz,\bz)_{33}=\del_2\dbar_2 K^{(\a,\b,\g)}(\bz,\bz)$. Thus, we have

$$K_q(\bz,\bz)|_{\Delta}=JK^{(\a,\b,\g)}(\bz,\bz)|_{\Delta}.$$
\end{rem}

\section{Homogeneity of multiplication operators of quotient modules}\label{hom}

In this section, assuming that the $m$-tuple of multiplication operators $\M$ on a reproducing kernel Hilbert module $\HM$ over $\A(\O)$ is homogeneous \w a subgroup $G$ of the biholomorphic automorphism group $\aut(\O)$ of $\O$, we show that the compression of $\M$ onto the quotient module $\HM_q$ is homogeneous under the action of some subgroup (depending upon $G$) of $\aut(\O)$. We make use of this fact to prove that the multiplication operators on quotient modules $\H^{(\l)}_q := \H^{(\l)} \ominus \H^{(\l)}_0$, $\l=(\l_1,\hdots,\l_m) \in \R^m$ with $\l_j> 0$, $1\leq j\leq m$, are homogeneous \w a certain subgroup of the automorphism group of $\mathsf{Z}$ where $\H^{(\l)}$ is the reproducing kernel Hilbert module over $\A(\D^m)$ with the reproducing kernel \beq\label{szrepk}K^{(\l)}(\bz,\bw):=(1-z_1\ov{w}_1)^{-{\l}_1}\cdots(1-z_m\ov{w}_m)^{-{\l}_m}\eeq on $\D^m$ and  the submodules $\H^{(\l)}_0 \subset \H^{(\l)}$ consist of holomorphic functions in $\H^{(\l)}$ vanishing of order $k$ along the connected complex submanifold $ \{(z_1,\hdots,z_m) \in \D^m : z_1=\cdots=z_d\}\subset \D^m$. 

%Thus, it exhibits a class of operators in $B_{Nr}(\D^{m-d+1})$ homogeneous under the identity component of the automorphism group of $\D^{m-d+1}$ where $N$ is the cardinality of the set $A=\{\a=(\a_1,\hdots,\a_d)\in(\N\cup\{0\})^d:|\a|<k\}$.
%
%\smallskip
%
Let $(\HM,\HM_q)\in\mathrm B_{r,k}(\O,\mathsf{Z})$ as introduced in Definition $\ref {CDq}$. Then $\HM$ gives rise to a {\h} $E_{\HM}$ with the global frame $\{K(.,\ov{w})\sigma_1,\hdots,K(.,\ov{w})\sigma_r:w \in \O^* \}$ on $\O^*$ where $\{\sigma_j:1\leq j\leq r\}$ is the standard ordered basis for $\C^r$. We denote this global holomorphic frame as $\f:=\{s_1(w),\hdots,s_r(w):w \in \O \}$ with $s_j(w):=K(.,\ov{w})\sigma_j$, $1 \leq j \leq r$ and $w \in \O$. Correspondingly, we have $\del^{\a}s_j(w)=\del^{\a}K(.,\ov{w})\sigma_j$, $1 \leq j \leq r$, $\a \in A$ and $w \in \O$. 

Following the procedure described in Section $5$ in \cite{OUIOQM}, we define the jet bundle $J^{(k)}E_{\HM}\overset{\pi_k}\ra \O$ of order $k$ of the holomorphic bundle $E_{\HM}\overset{\pi}\ra \O$ relative to the submanifold $\mathsf{Z}$ on $\O$ by declaring $\{\del^{\a}\f\}_{\a \in A}$ as holomorphic frame for $J^{(k)}E_{\HM}$ on $\O$, where we mean $\{\del^{\a}s_j\}_{j=1}^r$ by $\del^{\a}\f$. Since we have a global frame for $J^{(k)}E_{\HM}\overset{\pi_k}\ra \O$ we do not need to worry about the transition rule. We then define the hermitian metric on $J^{(k)}E_{\HM}\overset{\pi_k}\ra \O$ \w the frame $\{\del^{\a}\f\}_{\a \in A}$ by the Grammian $JH:=((JH_{\a\b}))_{\a,\b \in A}$ with $r\times r$ blocks $$JH_{\a\b}(w):=\left(\!\!\left(\langle\del^{\a} s_i(w), \del^{\b} s_j(w)\rangle\right)\!\!\right)_{i,j=1}^r\,\,\,\text{for}\,\,\a,\b \in A, w \in \O$$ where $H(w)=((\langle s_i(w),s_j(w)\rangle_{E_{\HM}}))_{i,j=1}^r$ is the metric on $E_{\HM}$ over $\O$.

We now give the description of subgroups of the biholomorphic automorphism group $\aut(\mathsf{Z})$ of $\mathsf{Z}$ which is of interest. Let $\Phi \in \aut(\O)$ be a biholomorphic automorphism of $\O$ such that $\Phi(\mathsf{Z})=\mathsf{Z}$. Then we note that $\Phi|_{\mathsf{Z}} \in \aut(\mathsf{Z})$ and we consider the subgroup $\aut(\O,\mathsf{Z}) \subset \aut(\O)$ which, by definition, is $$\aut(\O,\mathsf{Z}):= \{\Phi \in \aut(\O): \Phi(\mathsf{Z})=\mathsf{Z}\}.$$
Thus, $\aut(\O,\mathsf{Z})$ gives rise to a subgroup $\aut(\O,\mathsf{Z})|_{\mathsf{Z}}$ of $\aut(\mathsf{Z})$ defined by
$$\aut(\O,\mathsf{Z})|_{\mathsf{Z}}:= \{\phi \in \aut(\mathsf{Z}): \phi = \Phi|_{\mathsf{Z}} \text{ for some } \Phi \in \aut(\O,\mathsf{Z})\}.$$
In this set up, we are about to define vector bundle morphism between two vector bundles and the notion of homogeneous vector bundle \w some subgroup of the automorphism group of the base manifold. We recall few definitions following \cite{ACHOCD}.

\begin{defn}\label{vecmor}
Let $E \overset{\pi}{\ra} \O$ and $F \overset{\rho}{\ra} \O$ be two holomorphic vector bundles over $\O$. Then a vector bundle morphism is a pair of holomorphic mappings $(\hat{f},f)$ with $\hat{f}: E \ra F$ and $f : \O \ra \O$ such that the following diagram commutes 
$$\begin{tikzcd}
E\arrow[dd, "\pi"'] \arrow[rr, "\hat{f}"] &  & F \arrow[dd, "\rho"] \\
 &  &  \\
\O\arrow[rr, "f"'] &  & \O
\end{tikzcd}$$
that is, $\rho \circ\hat{f}=f \circ\pi$ and, for each $w \in \O$, $\hat{f}|_{w}:E_w \ra F_{f(w)}$ is a linear map where $E_w:=\pi^{-1}\{w\}$ and $F_{f(w)}:=\rho^{-1}\{f(w)\}$.
\end{defn}

An isomorphism between two holomorphic vector bundles, $E\overset{\pi}{\ra}\O$ and $F\overset{\rho}{\ra}\O$, is a vector bundle morphism so that both $\hat{f}:E \ra F$ and $f : \O \ra \O$ are biholomorphisms as well as, for each $w \in \O$, $\hat{f}|_w : E_w \ra F_{f(w)}$ is an invertible linear map. We denote $\aut(E)$ as the group of all automorphisms of the vector bundle $E\overset{\pi}\ra\O$. 

\begin{rem}
We note that the this is equivalent to saying that $\hat{f}$ is a bundle isomorhism from the vector bundle $E\overset{\pi}{\ra}\O$ onto the the bundle $f^*F\overset{\rho}{\ra}\O$ obtained by pulling back the bundle $F\overset{\rho}{\ra}\O$ via the map $f:\O\ra\O$. In other words, the following diagram commute.
$$\begin{tikzcd}
E\arrow[dd, "\pi"'] \arrow[rr, "\hat{f}"] &  & f^*F \arrow[lldd, "\rho"] \\
 &  &  \\
\O &  & 
\end{tikzcd}$$
\end{rem}

\begin{defn}\label{defhom}
Let $E \overset{\pi}{\ra}\O$ be a holomorphic vector bundle over $\O$ and $G \subset \aut(\O)$ be a subgroup of the group $\aut(\O)$ of biholomorphic automorphisms of $\O$. We then say $E \overset{\pi}{\ra}\O$ is homogeneous under the action of $G$ on $\O$ from left if $G$ acts on $\O$ transitively from left and, for every $g \in G$, there exists a bundle isomorphism $\hat{g}$ on $E \overset{\pi}{\ra}\O$ such that $\pi \circ \hat{g}=g \circ \pi$. 
\end{defn}

In view of the remark above, it is seen that $E\overset{\pi}{\ra}\O$ is homogeneous \w some subgroup $G\subset\aut(\O)$ if an only if the vector bundles $E\overset{\pi}{\ra}\O$ and $g^*E\overset{\pi}{\ra}\O$ are isomorphic for every $g\in G$.

In the following theorem, it is shown that the $k$-th order jet bundle of $E\overset{\pi}{\ra}\O$ relative to a co-ordinate plane is homogeneous \w some suitable group (depending upon $G$) assuming that $E\overset{\pi}{\ra}\O$ is homogeneous \w $G\subset\aut(\O)$.

\begin{thm}\label{jethom}
Let $\O$ be a bounded domain in $\C^m$ containing the origin and $Z:=\{z=(z_1,\hdots,z_m)\in\O:z_1=\cdots=z_d=0\}$. Let $E\overset{\pi}{\ra}\O$ be a holomorphic vector bundle of rank $r$ with a global holomorphic frame $\{s_1,\hdots,s_r\}$ which is homogeneous under the action of some subgroup $G$ of $\aut(\O,Z)$.Then the jet bundle $J^{(k)}E|_{Z}\overset{\pi_k}{\ra}Z$ relative to the submanifold $Z$ is homogeneous under the action of the group $G_Z=\{\phi\in\aut(Z):\phi=\Phi|_Z, \Phi\in G\}$. Moreover, if $(\hat{\Phi},\Phi)$ is the bundle isomorphism of $E$ with $\Phi\in G$ and $\Phi|_{Z}=\phi$ then $(J^{(k)}\hat{\Phi}|_{Z},\phi)$ defined by 
\beq\label{jpro2}
J^{(k)}\hat{\Phi}|_{Z}(\z(z)):= \sum_{\a \in A}\sum_{j=1}^r a_{\a j}(z)\del^{\a}(\hat{\Phi}\circ s_j(z)),
\eeq
for any holomorphic section $\z: Z \ra J^{(k)}E|_{Z}$ with $\z(z) = \sum_{\a \in A}\sum_{j=1}^r a_{\a j}(z)\del^{\a} s_j(z)$, is a bundle isomorphism for $J^{(k)}E|_{Z}$ where $a_{\a j}$ are holomorphic functions on $Z$ and $A=\{\a=(\a_1,\hdots,\a_m)\in (\N\cup\{0\})^d:|\a|<k\}$.
\end{thm}

\begin{proof}
Let $\phi\in G_Z$ be a biholomorphic automorphism of $Z$ and note from the definition of $G_Z$ that there exists $\Phi\in G$ with $\Phi|_{Z}=\phi$. Since $E\overset{\pi}{\ra}\O$ is homogeneous \w the group $G$, $\Phi$ gives rise to a bundle automorphism $(\hat{\Phi},\Phi)\in\aut(E)$. We now consider the mapping $J^{(k)}\hat{\Phi}|_{Z}$ as given in \eqref{jpro2} and observe that $(J^{(k)}\hat{\Phi}|_{Z},\phi)$ is a bundle morphism of the holomorphic vector bundle $J^{(k)}E|_{Z}\overset{\pi_k}{\ra} Z$. Indeed, for a holomorphic section $\z: Z \ra J^{(k)}E|_{Z}$ with $\z = \sum_{\a \in A}\sum_{j=1}^r a_{\a j}\del^{\a} s_j$ and $z \in Z$, we have
\Bea
\pi_k(J^{(k)}\hat{\Phi}|_{Z}(\z(z))) &=& \pi_k\left(\!\sum_{\a \in A}\sum_{j=1}^r a_{\a j}\del^{\a}\left(\hat{\Phi}\circ s_j(z)\right)\!\right)\\ 
&=& \pi_k\left(\!\sum_{\a \in A}\sum_{j=1}^r a_{\a j}\del^{\a}\left(\sum_{i=1}^r\hat{\Phi}(z)_{ij} s_i(\phi(z))\right)\!\right)\\ 
&=& \phi(z)\\
&=& \phi \circ \pi_k(\z(z))
\Eea
verifying that $\pi_k \circ J^{(k)}\hat{\Phi}|_{Z}=\phi \circ \pi_k$. Also, from the definition of $J^{(k)}\hat{\Phi}|_{Z}$ given in \eqref{jpro2} it is seen that $J^{(k)}\hat{\Phi}|_{Z}(\z(z))$ is holomorphic on $Z$ and consequently, so is the map $(z,\z(z))\mapsto (\phi(z),J^{(k)}\hat{\Phi}|_{Z}(\z(z)))$.

Next, we show that $J^{(k)}\hat{\Phi}|_Z$ is a bundle isomorphism. Let $(F_j,f_j)$, $j=1,2$ be two maps defined as follows.
$$F_1:E\ra E~~~\text{by}~F_1(s_j(p))=\sum_{i=1}^r \hat{\Phi}(p)_{ij}s_i(p)~~\text{with}~~f_1:\O\ra \O~~~\text{by}~~f_1(z)=z, \text{ and }$$ $$F_2:E\ra E~~~\text{by}~F_2(s_j(p))=s_j(p)~~\text{with}~~f_2:\O\ra \O~~~\text{by}~~f_1(z)=\Phi(z),$$  for $1 \leq j \leq r$, where $((\hat{\Phi}(p)))_{i,j=1}^r$ is the matrix of the linear mapping $\hat{\Phi}|_{E_p}:E_p \ra E_{\Phi(p)}$ \w the global frame $\{s_j\}_{j=1}^r$. It is then evident that both $(F_1,f_1)$ and $(F_2,f_2)$ are bundle morphisms. Furthermore, it follows, for $1 \leq j \leq r$, that
\Bea
F_2\circ F_1(s_j(p)) &=& F_2\left(\sum_{i=1}^r \hat{\Phi}(p)_{ij} s_i(p)\right)\\
&=& \sum_{i=1}^r \hat{\Phi}(p)_{ij}F_2(s_i(p))\\
&=&  \sum_{i=1}^r \hat{\Phi}(p)_{ij}s_i(\Phi(p))\\
&=& \hat{\Phi}(s_j(p))
\Eea
verifying that $F_2 \circ F_1 = \hat{\Phi}$. Note from the definition of $(F_1,f_1)$ that $F_1$ is a vector bundle isomorphism on $E\overset{\pi}{\ra}\O$ and hence, so is $F_2$ since $\hat{\Phi}\in\aut(E)$.

For a holomorphic section $\z: Z \ra J^{(k)}E|_{Z}$ with $\z = \sum_{\a \in A}\sum_{j=1}^r a_{\a j}\del^{\a} s_j$, we consider following bundle morphisms on $J^{(k)}E|_{Z}\overset{\pi_k}{\ra}Z$,
$$J^{(k)}F_1|_{Z}(\z(z))=\sum_{\a \in A}\sum_{j=1}^r a_{\a j}(z)\del^{\a}(F_1\circ s_j(z)), \text{ and }$$
$$J^{(k)}F_2|_{Z}(\z(z))=\sum_{\a \in A}\sum_{j=1}^ra_{\a j}(z)\del^{\a} (F_2\circ s_j(z)),$$
and observe that
$$J^{(k)}F_2|_{Z} \circ J^{(k)}F_1|_{Z}=J^{(k)}\hat{\Phi}|_{Z}.$$
Indeed, it follows from the definitions of $^{(k)}F_2|_{Z}$ and $J^{(k)}F_1|_{Z}$ that
\Bea
J^{(k)}F_2|_{Z} \circ J^{(k)}F_1|_{Z}(\z(z)) &=& J^{(k)}F_2\left(\sum_{\a \in A}\sum_{j=1}^r a_{\a j}(z)\del^{\a}(F_1\circ s_j(z))\right)\\
&=& J^{(k)}F_2\left(\sum_{\a \in A}\sum_{j=1}^r a_{\a j}(z)\del^{\a}\left(\sum_{i=1}^r \hat{\Phi}(z)_{ij}s_i(z)\right)\right)\\
&=& J^{(k)}F_2\left(\sum_{\a \in A}\sum_{j=1}^r a_{\a j}(z)\left(\sum_{\b \in A}\sum_{i=1}^r\J(\hat{\Phi}_{ij})_{\b\a}(z)\del^{\b}s_i(z)\right)\right)\\
&=& J^{(k)}F_2\left(\sum_{\a,\b \in A}\sum_{i,j=1}^r a_{\a j}(z)\J(\hat{\Phi}_{ij})_{\b\a}(z)\del^{\b}s_i(z)\right)\\
&=& \sum_{\a,\b \in A}\sum_{i,j=1}^r a_{\a j}(z)\J(\hat{\Phi}_{ij})_{\b\a}(z)\del^{\b}(F_2 \circ s_i(z))\\
&=& \sum_{\a,\b \in A}\sum_{i,j=1}^r a_{\a j}(z)\J(\hat{\Phi}_{ij})_{\b\a}(z)\del^{\b}(s_i\circ\phi(z))\\
&=& J^{(k)}\hat{\Phi}|_{Z}(\z(z)).
\Eea

Now it remains to show that the bundle morphism $J^{(k)}\hat{\Phi}|_{Z}$ is an isomprhism of the vector bundle $J^{(k)}E|_{Z}\overset{\pi_k}{\ra}Z$. We show this by showing that the bundle morphisms $J^{(k)}F_1$ and $J^{(k)}F_2$ are so. In this regard we observe, with the help of the Leibniz rule, that the matrix of the linear mapping $$(J^{(k)}F_1|_{Z})|_{J^{(k)}E_p}:J^{(k)}E_p\ra J^{(k)}E_p$$ \w the ordered basis $\{\del^{\a}s_1(p),\hdots,\del^{\a}s_r(p):\a \in A\}$ is the block lower triangular matrix 
\[
   J(\hat{\Phi})(p)=
   \begin{pmatrix*}[c]
     ((\hat{\Phi}(p)_{ij})){{}_{i,j=1}^r}  & \hdots & \hdots  &  \hdots   &  \hdots    & ((\hat{\Phi}(p)_{ij})){{}_{i,j=1}^r}    \\
         & \ddots &         &           &            & \vdots \\
         &        & \ddots  & \multicolumn{2}{c}{((\partial^{\alpha}\partial^{\beta}\hat{\Phi}\mathrlap{(p)_{ij})){}_{i,j=1}^r}} & \vdots \\
   \hspace*{2.0em}\smash{\text{\Huge0}} &        &         & \ddots    &            & \vdots \\
         &        &         &           & \ddots     & \vdots \\
         &        &         &           &            & ((\partial_d^{k-1}\hat{\Phi}(p)_{ij}))_{i,j=1}^r     \\
\end{pmatrix*}
\]
for $p\in Z$, where $A=\{\a \in(\N\cup\{0\})^d:|\a|\leq k-1\}$ and the order is obtained from the graded colexicographic ordering on $A$. Since the diagonal blocks of the above matrix is invertible $J^{(k)}F_1$ is an isomorphism. 

\smallskip

We now show that $J^{(k)}F_2|_{Z}$ is also an isomorphism. Since $\Phi(Z)=Z$ and $Z=\{z=(z_1,\hdots,z_m)\in \O:z_1=\cdots=z_d=0\}$ we have, for $z\in Z$, that 
$$\Phi(0,z_{d+1},\hdots,z_m)=(0,\Phi_{d+1}(0,z_{d+1},\hdots,z_m),\hdots,\Phi_m(0,z_{d+1},\hdots,z_m))$$ where by $0$ we mean the zero vector of $\C^d$. Thus, the functions $\Phi_1,\hdots,\Phi_d$ are constant on $Z$. 
Therefore, the Jacobian matrix $D\Phi(z)$ at any point $z\in Z$ becomes the following block triangular matrix 
\[ D\Phi(z)=
    \left(
    \!\!\begin{array}{cc}
     \left(\!\!\left(\dfrac{\partial}{\partial z_j}\Phi_i(z)\right)\!\!\right)_{1\leq i,j\leq d} & \text{\huge0}\\
         \left(\!\!\left(\dfrac{\partial}{\partial z_j}\Phi_i(z)\right)\!\!\right)_{d+1\leq i\leq m,1\leq j\leq d} &  \left(\!\!\left(\dfrac{\partial}{\partial z_j}\Phi_i(z)\right)\!\!\right)_{d+1\leq i,j\leq m}\\    
    \end{array}\!\!
    \right).
\]
Since $D\Phi(z)$ is invertible both the matrices $$\left(\!\!\left(\dfrac{\partial}{\partial z_i}\Phi_j(z)\right)\!\!\right)_{1\leq i,j\leq d} \text{ and     }\,\,\, \left(\!\!\left(\dfrac{\partial}{\partial z_i}\Phi_j(z)\right)\!\!\right)_{d+1\leq i,j\leq m}$$ are invertible. Therefore, it follows from Proposition 3.10 in \cite{OUIOQM} that $J^{(k)}F_2|_{Z}$ is invertible. As a consequence, we have that the matrix of $J^{(k)}\hat{\Phi}|_{Z}(z)$ \w the ordered frame  $\{\del^{\a}s_1(z),\hdots,\del^{\a}s_r(z):\a \in A\}$ is invertible for every point $z\in Z$. Thus, $J^{(k)}\hat{\Phi}|_{Z}:J^{(k)}E|_{Z}\ra J^{(k)}E|_{Z}$ is a bundle morphism whose matrix \w an ordered holomorphic frame is invertible verifying that $J^{(k)}\hat{\Phi}|_{Z}$ is an isomorphism on the vector bundle $J^{(k)}E|_{Z}\overset{\pi}{\ra}Z$. 
\end{proof}

In the rest of this section, we use the previous theorem in studying homogeneity of the compression of the tuple of multiplication operators on a quotient Hilbert module possessing a reproducing kernel. Recall that a reproducing kernel $K$ on $\O$ is said to be \textit{quasi-invariant} if for every $\bf{g}\in G$, there exists a holomorphic mapping $J_{\bf{g}}:\O \ra \text{GL}(r,\C)$ such that 
\beq\label{quasiinv}K(z,w) &=& J_{\bf{g}}(z)K({\bf{g}} z,{\bf{g}} w)J_{\bf{g}}(w)^*, ~ \text{for all }{\bf{g}} \in G ~ \text{and} ~ z,w \in \O.\eeq
We say a continuous mapping $J:G\times\O \ra \text{GL}(r,\C)$ defined by $J(\bf{g},z):=J_{\bf{g}}(z)$ is a \textit{cocycle} if it satisfies 
$$J_{\bf{gh}}(z) = J_{\bf{h}}(z)\cdot J_{\bf{g}}({\bf{h}}z), ~ \text{for all}~ \bf{g},\bf{h} \in G ~ \text{and} ~ z \in \O.$$ 
Note that the continuous mapping $J$ -- which is holomorphic in the second variable -- is a cocycle if and only if the action $U: G \times \text{Hol}(\O,\C^r) \ra \text{Hol}(\O,\C^r)$ defined by $$(U_{\bf{g}^{-1}}f)(z) := J_{\bf{g}}(z)f({\bf{g}}z)$$ defines an unitary representation of $G$ onto $\H_K$, provided $U(\H_K)\subset \H_K$. We now recall following theorem about equivalent conditions for homogeneous operators from \cite{OICHO}.

\begin{thm}\cite[Theorem 3.1]{OICHO}\label{homeqac}
Suppose that $\H_K$ is a {\rkhs} with an $r \times r$ matrix valued reproducing kernel $K$ on  $\O$ on which the multiplicatin operator $\bf{M}=(M_1,\hdots,M_m)$ is bounded. Also, assume that either ${\bf{M}}^* \in \mathrm B_r(\O^*)$, or, $\C[{\bf{z}}]\otimes \C^r\subset\H_K$ is dense. Then following are equivalent.
\begin{itemize}
\item[(a)] The $m$-tuple $\bf{M}$ is $G$-homogeneous.
\item[(b)] The reproducing kernel $K$ is quasi-invariant, that is, it transforms according to the rule 
$$K(z,w) = J_{\bf{g}}(z)K({\bf{g}}z,{\bf{g}}w)J_{\bf{g}}(w)^*, ~ \text{for all }\bf{g} \in G ~ \text{and} ~ z,w \in \O$$ for some continuous function $J:G\times\O \ra \text{GL}(r,\C)$ holomorphic in the second variable denoted by $J(\textbf{g},z)= J_{\bf{g}}(z)$.
\item[(c)] The operator $U_{\bf{g}^{-1}}:\H_K \ra \H_K$ is well defined as well as is unitary for every $\bf{g} \in G$ where $(U_{\bf{g}^{-1}}f)(z) := J_{\bf{g}}(z)f(\bf{g}z)$ for $z \in \O$.
\end{itemize}
\end{thm}

Suppose now that $\HM$ is a reproducing kernel Hilbert module with the reproducing kernel $K$ defined on $\O^*:=\{\ov{w}:w\in \O\}$ so that the adjoint ${\M}^*=(M_1^*,\hdots,M_m^*)$ of the tuple of multiplication operators ${\M}=(M_1,\hdots,M_m)$ is in $\mathrm B_r(\O)$. It turns out that the associated {\h} $E_{\HM} \overset{\pi}{\ra}\O$ to $\M$ possesses  a global holomorphic frame $\f:=\{s_1(w),\hdots,s_r(w): w \in \O\}$ where $s_j(w)= K(.,\ov{w})\si_j$, for $j=1,\hdots,r$, $w \in \O$ and $\{\si_j\}_{j=1}^r$ are the standard order basis for $\C^r$. Then the hermitian metric $H$ on $E_{\HM}$ \w this frame becomes 
$$H_w(s_j(w),s_i(w))=\<K(.,\ov{w})\si_j,K(.,\ov{w})\si_i\>_{\HM}=\langle K(\ov{w},\ov{w})\si_j,\si_i\rangle_{C^r} = K_{ij}(\ov{w},\ov{w}),~\text{for}~ w\in \O.$$ 
We relate the homogeneity of this vector bundle $E_{\HM}$ to Theorem \ref{homeqac} in the theorem below.

\begin{thm}\label{homeqgc}
For any subgroup $G$ of the automorphism group $\aut(\O)$ of $\O$, the following are equivalent.
\begin{enumerate}
\item[(i)] $E_{\HM}$ is homogeneous under the action of the subgroup $G$ of $\aut(\O)$.
\item[(ii)] The reproducing kernel $K$ is quasi-invariant \w the subgroup, $G_c:=\{g_c=c\circ g \circ c^{-1}: g \in G\}$ of $\aut(\O^*)$.

\end{enumerate}
\end{thm}

\begin{proof}
Assume that $E_{\HM} \overset{\pi}{\ra}\O$ is homogeneous under the action of $G$, that is, for each $g\in G$, there exists a bundle automorphism $\hat{g}:E_{\HM}\ra E_{\HM}$. In other words, for each $w \in \O$, there is a linear isometry $\hat{g}|_w:E_{\HM}|_w \ra E_{\HM}|_{gw}$. Therefore, it follows that 
\Bea
H_w(s_j(w),s_i(w)) &=& H_{gw}(\hat{g}|_w(s_j(w)),\hat{g}|_w(s_i(w)))\\
&=& H_{gw}\left(\sum_{k=1}^r[\hat{g}|_w]_{kj}s_k(gw),\sum_{l=1}^r[\hat{g}|_w]_{li}s_l(gw)\right)\\
&=& \sum_{l,k=1}^r[\hat{g}|_w]_{kj}\ov{[\hat{g}|_w]_{li}}H_{gw}(s_k(gw),s_l(gw)).
\Eea 
Thus, using the definition of the metric $H_{gw}$ in terms of reproducing kernel we get from the above equation 
\Bea
K_{ij}(\ov{w},\ov{w}) &=& \<K(\ov{w},\ov{w})\si_j,\si_i)\>\\
&=& \sum_{l,k=1}^r[\hat{g}|_w]_{kj}\ov{[\hat{g}|_w]_{li}}\<K(\ov{gw},\ov{gw})\si_k,\si_l\>\\
&=& \sum_{l,k=1}^r [\hat{g}|_w]^*_{il}K_{lk}(\ov{gw},\ov{gw})[\hat{g}|_w]_{kj}
\Eea
which is equivalent to the following identities
\bea\label{qinv1}
\nonumber K(\ov{w},\ov{w}) &=& [\hat{g}|_w]^*K(\ov{gw},\ov{gw})[\hat{g}|_w]\\
&=& [\hat{g}|_w]^*K(g_c(\ov{w}),g_c(\ov{w}))[\hat{g}|_w].
\eea
For $z \in \O^*$, the equation above takes the form 
$$K(z,z) = [\hat{g}|_{\ov{z}}]^*K(g(z),g(z))[\hat{g}|_{\ov{z}}]$$
leading us to define the matrix valued function $J_{g_c}:\O^* \ra \text{GL}(r,\C)$ by $J_{g_c}(z)=[\hat{g}|_{\ov{z}}]^*$. Evidently, $J_{g_c}$ is a holomorphic function on $\O^*$. In this set up, the equation $\eqref{qinv1}$, after polarization, turns out to be
\beq\label{qinv}
K(z,w) &=& J_{g_c}(z) K(g_cz,g_cw) J_{g_c}(w)^*, ~ \text{for}~z,w \in \O^*.
\eeq
Since $(\hat{h}\circ \hat{g})|_w = \hat{h}|_{g(w)}\circ \hat{g}|_w$ for each $w\in \O$, it follows that the mapping $J: G_c \times \O^* \ra \C^{r \times r}$ defined by $J(g_c,z)=J_{g_c}(z)$ satisfies the cocycle identity (cf. \eqref{quasiinv}) which together with the equation \eqref{qinv} imply that $K$ is a quasi invariant kernel completing the proof of (i) $\implies$ (ii). 

So it remains to show that (ii) $\implies$ (i). For $g_c \in G_c$, let $J_{g_c}:\O^* \ra \text{GL}(r,\C)$ be a holomorphic mapping satisfying 
$J_{g_ch_c}(z)=J_{g_c}(z) \cdot J_{h_c}(g_cz)$, $z \in \O^*$ such that 
$$K(z,w) = J_{g_c}(z) K(g_cz,g_cw)J_{g_c}(w)^*, \text{ for } z,w\in \O^*,\,\,g_c \in G_c.$$
Define, for $g \in G$, the bundle morphism $\hat{g}:E_{\HM} \ra E_{\HM}$ as follows:
$$\hat{g}s_j(w):=\sum_{k=1}^r [J_{g_c}(\ov{w})^*]_{kj}s_k(gw),~ ~ \text{for}~w \in \O.$$
Since $w \mapsto s_j(w)$, $w \mapsto J_{g_c}(\ov{w})^*$ as well as $w \mapsto gw$ are all holomorphic on $\O$, so is the mapping $w \mapsto \hat{g}|_w$. Therefore, $\hat{g}$ is a holomorphic bundle morphism which preserves the hermitian structure of $E_{\HM} \ra \O$. Thus, it completes the proof.
\end{proof}

\begin{thm}\label{homeqc}
Let $\O \subset \C^m$ be a bounded domain and $\HM$ be a reproducing kernel Hilbert module in $\mathrm B_r(\O)$ with the reproducing kernel $K$ defined on $\O^*$. Let $G \subset \aut(\O)$ be a subgroup. Then the {\h} $E_{\HM}\ra \O$ is homogeneous \w $G$ if and only if the multiplication operator ${\M}=(M_1,\hdots,M_m)$ on $\HM$ is $G_c$-homogeneous where $G_c:=\{g_c=c\circ g \circ c^{-1}: g \in G\}$ and $c:\O \ra\O^*$ is the mapping defined by $z \mapsto \ov{z}$.
\end{thm}

\begin{proof}
The proof of this theorem follows from Theorem $\ref{homeqgc}$ and Theorem \ref{homeqac}.
\end{proof}

\begin{thm}\label{thmhom0}
Let $\O$ be a bounded domain in $\C^m$ containing the origin, $Z:=\{z=(z_1,\hdots,z_m)\in\O:z_1=\cdots=z_d=0\}$ and $(\HM,\HM_q)$ be a pair of Hilbert modules in $\mathrm B_{r,k}(\O,Z)$. Also, assume that $\textbf{M}=(M_1,\hdots,M_m)$ on $\HM$ is $G_c$-homogeneous where $G$ is a subgroup of $\aut(\O,Z)$, $c:\O \ra \O^*$ is the conjugation mapping and $G_c =\{g_c=c \circ g \circ c^{-1}: g \in G\}\subset \aut(\O^*,Z^*)$ with $Z^*=c(Z)$. Then the compression of  $\textbf{M}$ onto the quotient module $\HM_q$ is homogeneous \w the subgroup $(G_{Z})_c:=\{\phi_c : \phi \in G_{Z}\}$ where $G_{Z}=\{\phi \in \aut(Z):\phi = \Phi|_{Z},~\text{for some }\Phi \in G\}$.
\end{thm}

\begin{proof}
We begin by pointing out from Theorem \ref{homeqac} and Theorem \ref{homeqgc} that the vector bundle $E_{\HM} \overset{\pi} \ra \O$ associated to $\HM$ is homogeneous \w $G \subset \aut(\O,Z)$. 
%We see that the group $(G_{Z})_c$ is nothing else but the group 
%$$G_{Z^*}:=\{\phi \in \aut(Z^*):\phi = \Phi|_{Z^*}~\text{for some }\Phi\in G_c\}.$$ Indeed, since any element $\phi \in \aut(Z^*)$ is $\psi_c=c^{-1}\circ\psi\circ c$ for some $\psi \in \aut(Z)$, we have that if $\phi \in G_{Z^*}$ there exists $\psi \in \aut(Z)$ such that $\psi_c = \Phi|_{Z^*}$ for some $\Phi \in G_c$. Further, since every $\Phi \in G_c$ there exists $\Psi \in G$ so that $\Phi=\Psi_c$ we conclude that $\psi_c=\Psi_c|_{Z^*}$ for some $\Psi\in G$. This then shows that $G_{Z^*} \subset (G_{Z})_c$. Conversely, for $\phi_c \in (G_{\mathsf{Z}})_c$, we note that 
%$$\phi_c=c\circ\phi\circ c^{-1}=c\circ\Phi|_{Z}\circ c^{-1}=(c\circ \Phi \circ c^{-1})|_{Z^*}$$ and hence $\phi_c\in G_{Z^*}$.
It follows from Theorem $\ref{jethom}$ that the jet bundle $J^{(k)}E_{\HM}|_{Z}\ra Z$ is homogeneous \w $G_{Z}$ as a holomorphic vector bundle. We now show that $J^{(k)}E_{\HM}|_{Z}\ra Z$ is homogeneous \w $G_{Z}$ as a {\h}.

Let $\phi \in G_{Z}$ with $\phi=\Phi|_{Z}$ for some $\Phi \in G$ and $\hat{\Phi}$ be the bundle automorphism of $E_{\HM}\overset{\pi}\ra\O$ satisfying $\pi\circ\hat{\Phi}=\phi\circ\pi$. Following the proof of Theorem $\ref{jethom}$, note that $J^{(k)}\hat{\Phi}|_{Z}$ defined in $\eqref{jpro2}$ is a bundle automorphism of the holomorphic vector bundle $J^{(k)}E_{\HM}|_{Z}\overset{\pi_k}{\ra} Z$. So it is now enough to check that $J^{(k)}\hat{\Phi}|_{Z}$ preserves the hermitian metric $JH$ on $J^{(k)}E_{\HM}|_{Z}\overset{\pi_k}{\ra} Z$. To this extent, we take the global holomorphic frame $\{\del^{\a}\f\}_{\a\in A}$ of $J^{(k)}E_{\HM}|_{Z}\overset{\pi_k}\ra Z$ where $\f=\{s_1,\hdots,s_r\}$ is the global holomorphic frame for $E_{\HM}\overset{\pi}\ra \O$ with $s_j(w)=K(.,\ov{w})\si_j$, $1 \leq j \leq r$, $w \in \O$ and $A$ is the set of all multi-indices $\a=(\a_1,\hdots,\a_d)$ with $|\a| < k$. Now from the equation $\eqref{jpro2}$ we have , for $w \in Z$, that
\Bea
JH_{\phi(w)}(J^{(k)}\hat{\Phi}|_{Z}(\del^{\a}s_i|_w),J^{(k)}\hat{\Phi}|_{Z}(\del^{\b}s_j|_w)) 
&=& JH_{\phi(w)}(\del^{\a}(\hat{\Phi}|_{Z}\circ s_i(w)),\del^{\b}(\hat{\Phi}|_{Z}\circ s_j(w)))\\
&=& \del^{\a}\dbar^{\b}H_{\phi(w)}(\hat{\Phi}|_{Z}\circ s_i(w),\hat{\Phi}|_{Z}\circ s_j(w))\\
&=& \del^{\a}\dbar^{\b}H_w(s_i(w),s_j(w))\\
&=& JH_w(\del^{\a}s_i|_w,\del^{\b}s_j|_w)
\Eea 
verifying that $J^{(k)}\hat{\Phi}|_{Z}$ is a bundle automorphism of the {\h} $J^{(k)}E_{\HM}|_{Z}\overset{\pi_k}\ra Z$. In the above calculation third equality holds as $\hat{\Phi}$ is a bundle automorphism. 

Thus $J^{(k)}E_{\HM}|_{Z}\overset{\pi_k}\ra Z$ is homogeneous \w $G_{Z}$ as {\h} with the hermitian metric $JH(w)=JK(\ov{w},\ov{w})$ where $K$ is the reproducing kernel of $\HM$ and $JK$ is as in \eqref{repk}. Consequently, following the proof of the Theorem \ref{homeqgc} note that $JK|_{\text{res}Z^*}$ is quasi-invariant \w the group $G_{Z^*}$ which is nothing else but the group $(G_Z)_c$. Therefore, using Theorem $\ref{homeqac}$ we conclude that compression of the adjoint of the multiplication operator $J\textbf{M}$ on $J(\HM)|_{\text{res}Z^*}$ is homogeneous \w the group $(G_{Z})_c$. Then the proof follows from the fact that the compression of the tuple of multiplication operators $J\textbf{M}$ onto $J(\HM)|_{\text{res}Z^*}$ and the compression of the tuple of multiplication operators $\textbf{M}$ onto the quotient module $\HM_q$ are unitarily equivalent (cf. \cite[Theorem 4.5]{OUIOQM}).
\end{proof}

\begin{rem}
It is to be noted that homogeneity of the tuple of multiplication operators on the quotient space does not depend on the choice of jets provided the tuple of multiplication operators on the given Hilbert module is homogeneous.
\end{rem}

We now move on to the case of homogeneity of the tuple of multiplication operators on the quotient modules obtained from submodules of functions vanishing up to higher order along $\mathsf{Z}$. In this regard, we first prove that the homogeneity of the multiplication operators on such quotient space is independent of change of variables. Let $\theta:\O_1\ra \O_2$ be a biholomorphism between two bounded domains  $\O_1$ and $\O_2$ of $\C^m$, and $\HM_1$ be a reproducing kernel Hilbert module on $\O_1$ over $\A(\O_1)$ with an $r\times r$ matrix valued reproducing kernel $K_1$ on which the $m$-tuples of multiplication operators $\M=(M_{z_1},\hdots,M_{z_m})$ by co-ordinate functions are bounded. Define $$\HM_2:=\theta^*(\HM_1):=\{f\circ \theta^{-1}:f\in\HM_1\}.$$ It is then evident that $\HM_2$ is a reproducing kernel Hilbert module over $\A(\O_2)$ with 
\beq\label{K_2}K_2(\tilde{z},\tilde{w})=K_1(\theta^{-1}(\tilde{z}),\theta^{-1}(\tilde{w})),~~~\text{for}~~~\tilde{z},\tilde{w}\in\O_2\eeq as the reproducing kernel. In this set up, we prove the following proposition showing that $K_1$ is quasi-invariant \w some subgroup $G$ of the automorphism group of $\O_1$ if and only if so is $K_2$ \w the subgroup $G_{\theta}:=\{\Phi_{\theta}=\theta\circ\Phi\circ\theta^{-1}:\Phi\in G\}$.

\begin{prop}\label{eqhom}
Let $\theta:\O_1\ra \O_2$ be a biholomorphism between two bounded domains  $\O_1$ and $\O_2$ of $\C^m$, and $\HM_1$ be a reproducing kernel Hilbert module on $\O_1$ over $\A(\O_1)$ with an $r\times r$ matrix valued reproducing kernel $K_1$ on which the $m$-tuples of multiplication operators $\M$ by co-ordinate functions are bounded. Also, assume that either $\M^*\in \mathrm B_r(\O_1^*)$, or, $\C[\bz]\otimes\C^r\subset \HM_1$ is dense. Then $\M$ is $G-$homogeneous \w some subgroup $G$ of automorphism group of $\O_1$ if and only if the $m$-tuple of multiplication operators by the co-ordinate functions on $\theta^*(\HM_1)$  is $G_{\theta}-$homogeneous where $G_{\theta}:=\{\theta\circ\Phi\circ\theta^{-1}:\Phi\in G\}$.
\end{prop}

\begin{proof}
Suppose that $\M$ is $G-$homogeneous. Consequently, $K_1$ is quasi-invariant \w $G$, that is, there is a co-cycle $J_1:G\times \O_1\ra \C^{r\times r}$ such that
\beq\label{cocycle K_1}K_1(z,w)=J_1(\Phi,z)K_1(\Phi(z),\Phi(w))J_1(\Phi,w)^*,\eeq for $z,w\in \O_1$ and $\Phi\in G$. We then have from the definition of $K_2$ in \eqref{K_2} that
$$K_2(\tilde{z},\tilde{w})  =  J_1(\Phi,\theta^{-1}(\tilde{z}))K_2(\Phi_{\theta}(\tilde{z}),\Phi_{\theta}(\tilde{w}))J_1(\Phi,\theta^{-1}(\tilde{w}))^*$$ 
where we denote $\theta\circ\Phi\circ\theta$ by $\Phi_{\theta}$. Now define the function $J_2:G_{\theta}\times \O_2\ra\C^{r\times r}$ by $$J_2(\Phi_{\theta},\tilde{z}):=J_1(\Phi,\theta^{-1}(\tilde{z}))$$ and observe that $J_2$ is a cocycle for the group $G_{\theta}$. Indeed, Since $J_1$ is a cocycle and $\theta$ is a biholomorphism,  $J_2(\Phi_{\theta},.)$ is holomorphic for each $\Phi_{\theta}\in G_{\theta}$ and for $\Phi_{\theta},\Psi_{\theta}\in G_{\theta}$ and $\tilde{z}\in \O_2$, satisfies the co-cycle identity:
$$J_2(\Phi_{\theta}\Psi_{\theta},\tilde{z}) = J_1(\Phi\Psi,\theta^{-1}(\tilde{z})) = J_2(\Psi_{\theta},\tilde{z})J_2(\Phi_{\theta},\Psi_{\theta}(\tilde{z}))$$
where the last equality follows from \eqref{cocycle K_1}. Hence $K_2$ is quasi-invariant verifying that the $m$-tuple of multiplication operators by co-ordinate functions on $\theta^*(\HM)$ is $G_{\theta}-$homogeneous.

Conversely, let $K_2$ be quasi-invariant \w $G_{\theta}$. Then there exists a co-cycle $J_2:G_{\theta}\times \O_2\ra\C^{r\times r}$ such that
$$K_2(\tilde{z},\tilde{w})=J_2(\Phi_{\theta},\tilde{z})K_2(\Phi_{\theta}(\tilde{z}),\Phi_{\theta}(\tilde{w}))J_2(\Phi_{\theta},\tilde{w})^*.$$ Then a similar argument as above with $J_1(\Phi,z)=J_2(\Phi_{\theta},\theta{z})$, for $\Phi\in G$ and $z\in \O_1$, yields that $K_1$ is quasi-invariant \w the group $G$. As a consequence, with the help of Theorem $\ref{homeqac}$ we have that $\M$ is $G-$homogeneous.
\end{proof}

\begin{rem}\label{z^*}
For a bounded domain $\O \subset \C^m$, note that if $\mathsf{Z}\subset \O$ is a connected complex submanifold of codimension $d$, then so is the subset $\mathsf{Z}^*:=c(\mathsf{Z})=\{\ov{z}:z \in \mathsf{Z}\}$ of $\O^*$. Indeed, take any point $\ov{p} \in \mathsf{Z}^*$ with $p \in \mathsf{Z}$. Since $\mathsf{Z}$ is a complex submanifold of codimension $d$ in $\O$ there exists a neighbourhood $U$ around $p$ in $\O$ and $d$ holomorphic functions $f_1,\hdots,f_d$ on $U$ such that $\mathsf{Z}\cap U=\{z \in U:f_j(z)=0, ~ 1 \leq j \leq d\}$. It is then evident that $\mathsf{Z}^*\cap U^*=\{\ov{z} \in U^*:g_j(\ov{z})=0, ~ 1 \leq j \leq d\}$ where $U^*=c(U)$ and $g_j=\ov{f \circ c^{-1}}$, $1 \leq j \leq d$. So $\mathsf{Z}^*$ is a complex submanifold as $g_j$'s are holomorphic on $U^*$ for $1 \leq j \leq d$. Furthermore, since $c$ is a homeomorphism and $\mathsf{Z}$ is connected so is $\mathsf{Z}^*$.
\end{rem}

\begin{thm}\label{thmhom}
Let $\O$ be a bounded domain in $\C^m$ and $\mathsf{Z}\subset \O$ be a connected complex submanifold of codimension $d \geq 1$ such that there exists a biholomorphism $\theta:\O\ra\C^m$ onto its image of the form $\theta(z)=(\theta_1(z),\hdots,\theta(z_d),z_{d+1},\hdots,z_m)$ such that $\theta(\mathsf{Z})=Z$ where $Z$ is the co-ordinate plane $\{\l=(\l_1,\hdots,\l_m)\in\theta(\O):\l_1=\cdots=\l_d=0\}$. Let $(\HM,\HM_q)$ be a pair of Hilbert modules in $\mathrm B_{r,k}(\O,\mathsf{Z})$ and $G$ be a subgroup of $\aut(\O,\mathsf{Z})$. Also, assume that the $m$-tuple of multiplication operators $\textbf{M}$ by co-ordinate functions on $\HM$ is $G_c$-homogeneous where $c:\O \ra \O^*$ is the conjugation mapping and $G_c =\{g_c=c \circ g \circ c^{-1}: g \in G\}\subset \aut(\O^*,\mathsf{Z}^*)$ with $\mathsf{Z}^*=c(\mathsf{Z})$. Then the compression of the multiplication operator $\textbf{M}$ onto the quotient module $\HM_q$ is homogeneous \w the subgroup $(G_{\mathsf{Z}})_c:=\{\phi_c : \phi \in G_{\mathsf{Z}}\}$ where $G_{\mathsf{Z}}=\{\phi \in \aut(\mathsf{Z}):\phi = \Phi|_{\mathsf{Z}},~\text{for some }\Phi \in G\}$.
\end{thm}

\begin{proof}
We begin by pointing out from the Remark \ref{z^*} that $c(\mathsf{Z})=\mathsf{Z}^*$ is a connected complex submanifold of $\O^*=c(\O)$, and $\theta_c=c\circ\theta\circ c^{-1}$ is the corresponding biholomorphism sending $\O^*$ onto $\theta(\O)^*$ as well as $\mathsf{Z}^*$ onto $Z^*=c(Z)$.

Since $\M$ is $G_c-$homogeneous it follows from Proposition \ref{eqhom} that so is the $m$-tuple of multiplication operators by the co-ordinate functions on $\theta_c^*(\HM)$ under the action of the group $(G_{\theta})_c$. Moreover, $\theta(\O)^*$ contains the origin and $Z^*$ is also the co-ordinate plane $\{\ov{\l}\in\theta(\O)^*:\l_1=\cdots=\l_d=0\}$. It follows from \cite[Proposition 3.10]{OUIOQM} that the quotient module $\HM_q$ is unitarily equivalent to $\theta_c^*(\HM_q)$ as Hilbert modules which happens to be the quotient module corresponding to the submodule of functions in $\theta_c^*(\HM)$ vanishing of order $k$ along the co-ordinate plane $Z^*$. Consequently, the pair $((\theta_c^*(\HM)),\theta_c^*(\HM_q))$ is in $\mathrm B_{r,k}(\theta(\O),Z)$ since $(\HM,\HM_q)$ is in $\mathrm B_{r,k}(\O,\mathsf{Z})$. Therefore, it follows from Theorem \ref{thmhom0} that the compression of the multiplication operator onto the quotient module $\theta_c^*(\HM_q)$ is $((G_{\theta})_Z)_c-$homogeneous where $((G_{\theta})_Z)_c=\{\phi\in \aut(Z):\phi=\Phi_{\theta}|_{Z},\Phi_{\theta}\in G_{\theta}\}$. We note that the subgroup $((G_{\theta})_Z)_c$ is nothing else but the group $((G_{\mathsf{Z}})_{\theta})_c$. So the proof follows from Proposition \ref{eqhom}.
\end{proof}

Let $\a=(\a_1,\hdots,\a_m)$ with $\a_j > 0$, $1 \leq j \leq m$ and $\H^{(\a)}(\D^m)$ be the weighted Bergman module over $\D^m$ with the reproducing kernel 
$$K^{(\a)}(\textbf{z},\textbf{w})=\prod_{j=1}^m(1-z_j\ov{w}_j)^{-\a_j}, \text{ for }\textbf{z},\textbf{w} \in \D^m.$$
Let $\Delta_d$ be the submanifold  $\Delta_d:= \{\textbf{z} \in \D^m: z_1=\cdots=z_d\}$ of $\D^m$. Since $(\D^m)^*=\D^m$ and $\Delta_d^*=\Delta_d$ we identify $(\D^m)^*$ with $\D^m$ and $\Delta_d^*$ with $\Delta_d$ via the conjugation mapping $c:\D^m \ra (\D^m)^*$ defined by $\textbf{z} \mapsto \ov{\textbf{z}}$. So we consider the submodule $\H^{(\a)}_0$ consisting of functions in $\H^{(\a)}(\D^m)$ vanishing along $\Delta_d$ of order $k$. Then using the jet construction we have that the quotient module $\H^{(\a)}_q:=\H^{(\a)}(\D^m)\ominus\H^{(\a)}_0$ is module isomorphic to the module of $k$-jets, $J(\H^{(\a)}(\D^m))|_{\text{res}\Delta_d}$, relative to $\Delta_d$. It turns out that $(\H^{(\a)}(\D^m),\H^{(\a)}_q)$ is in $\mathrm B_{1,k}(\D^m,\Delta_d)$ (cf. \cite[Theorem 4.8]{OUIOQM}) . 

Let $\textbf{M}^{(\a)}=(M_{z_1},\hdots,M_{z_m})$ be the $m$-tuple multiplication operators by co-ordinate functions on $\H^{(\a)}$. Then the compression of $\textbf{M}^{(\a)}$ onto $\H^{(\a)}_q$ is unitarily equivalent to the multiplication operator on $J(\H^{(\a)})|_{\text{res}\Delta_d}$ (cf. \cite[Theorem 4.5]{OUIOQM}). From now on, we denote the compression of $\textbf{M}^{(\a)}$ onto $\H^{(\a)}_q$ as $\textbf{M}^{(\a)}_k$. The following corollary yields a special case of the previous theorem.  

\begin{cor}\label{corhom}
The operators $\textbf{M}^{(\a)}_k$ are homogeneous \w the group M\"ob$(\D)^{m-d+1}$, for $k \in \N$ and $\a=(\a_1,\hdots,\a_m)$ with $\a_j \geq 0$, $j=1,\hdots,m$.
\end{cor}

\begin{proof}
Since $\textbf{M}^{(\a)}$ on $\H^{(\a)}(\D^m)$ is M\"ob$^m$-homogeneous and $(\H^{(\a)}(\D^m),\H^{(\a)}_q)$ is in $\mathrm B_{1,k}(\D^m,\Delta_d)$, it follows from Theorem $\ref{thmhom}$ that $\textbf{M}^{(\a)}_k$ is $(G_{\Delta_d})_c$-homogeneous where $G_{\Delta_d}=\{\phi \in \text{Aut}_0(\Delta_d):\phi = \Phi|_{\Delta_d} \text{ for some } \Phi \in G\}$ and $G:=\{\Phi \in \text{M\"ob}(\D)^m:\Phi(\Delta_d)=\Delta_d\}$. Since $\Delta_d$ is biholomorphic to $\D^{m-d+1}$ via the biholomorphism $\textbf{z}\mapsto (z_1,z_{d+1},\hdots,z_m)$, it follows that the group $G_{\Delta_d}$ is the group M\"ob$(\D)^{m-d+1}$ and M\"ob$(\D)=$M\"ob$(\D)_c$ with the identification $\D =\D^*$. Hence this verifies the claim. 
\end{proof}

\section{Reducibility of quotient modules}\label{qmreducibility}

Let $\a=(\a_1,\hdots,\a_m)$ with $\a_j > 0$, $1 \leq j \leq m$ and $\H^{(\a)}(\D^m)$ be the weighted Bergman module over the disc algebra $\A(\D^m)$ with the reproducing kernel 
\beq\label{ker}K(\bz,\bw):= K^{(\a)}(\bz,\bw) = \prod_{j=1}^m(1-z_j\ov{w}_j)^{-\a_j},~~~\bz,\bw \in \D^m.\eeq
Then the tuple of multiplication operators $\M=(M_{z_1},\hdots,M_{z_m})$ on $\H^{(\a)}(\D^m)$ is in $\mathrm B_1(\D^m)$. So it corresponds the {\hl} $E \overset{\pi}{\ra}\D^m$ with the global holomorphic frame $s(\bw):=K(.,\ov{\bw}),\,\,\bw \in \D^m$. We denote $\H^{(\a)}(\D^m)$ by the letter $\HM$ and consider the submodule $\HM^{(n)}$ of functions in $\HM$ vanishing of order $n$ along the diagonal set $\Delta =\{\bz=(z_1,\hdots,z_m) \in \D^m: z_1=\hdots=z_m\}$. Let $\HM^{(n)}_q = \HM \ominus \HM^{(n)}$ be the quotient module. It then follows from the jet construction mentioned in Section $\ref{JC}$ that $\HM^{(n)}_q$ is unitarily equivalent to the module of jets, $J^{(n)}\HM|_{\text{res}\Delta}$, of order upto $n$ restricted to $\Delta$. Moreover, since $K$ is a reproducing kernel with the diagonal power series expansion, $\HM^{(n)}_q$ is in $B_{N}(\D)$ with $N$ is the cardinality of the set $A=\{\a=(\a_1,\hdots,\a_m)\in\R_{>0}^m:|\a|<n\}$, thanks to \cite[Theorem 4.10]{OUIOQM}, and consequently, $J^{(n)}\HM|_{\text{res}\Delta}$ corresponds to the {\h} $\mathcal{E}$ over $\D$. Let $J^{(n)}E|_{\Delta}\ra \Delta$ be  the jet bundle of order $n$ relative to $\Delta$ associated with the line bundle $E \overset{\pi}{\ra}\D^m$ with the global holomorphic frame 
$\{\del^{\a}s(w):|\a|<n\}$ for $w\in\Delta.$ These two vector bundles, namely $\mathcal{E}\ra \D$ and $J^{(n)}E|_{\Delta}\ra \Delta$, under the canonical identification of $\D$ and $\Delta$, are unitarily equivalent due to the fact that the jet map $J^{(n)}:\HM\ra J^{(n)}(\HM)$ introduced in Section \ref{JC} is a unitary module map sending the quotient module $\HM^{(n)}_q$ isomorphically onto $J^{(n)}\HM|_{\text{res}\Delta}$ as Hilbert modules. Therefore from now on, we consider $J^{(n)}\HM|_{\text{res}\Delta}$ as the vector bundle corresponding to the quoient module $J^{(n)}\HM|_{\text{res}\Delta}$.

In this section, we first prove that the compression of the multiplication operator $\M=(M_{z_1},\hdots,M_{z_m})$ on $\H^{(\a)}(\D^m)$, to the quotient module $\HM^{(n)}_q$ is reducible. We denote this operator as $M^{(\a)}_n$. Further, along the way we completely decompose this operator into irreducible factors and identify the irreducible factors as the \textit{Generalized Wilkins' Operators} (\cite[Page 428]{HOPRMGS}). In order to accomplish our goal, we begin with $m=3$. In this case, we first show that the jet bundle $J^{(n)}E|_{\Delta}\ra\Delta$ correponding to the quotient module $\HM^{(n)}_q$ is orthogonal direct sum of some subbundles which correspond the \textit{Generalized Wilkins' Operators}.

\begin{lem}\label{lem1}
Let $K$ be the reproducing kernel on $\D^3$ defined by 
$$K(\bz,\bw) = (1-z_1\ov{w}_1)^{-\a}(1-z_2\ov{w}_2)^{-\b}(1-z_3\ov{w}_3)^{-\g},~~~\bz,\bw \in \D^3.$$
Then, for any $l,j \in \N \cup \{0\}$ with $0 \leq j \leq l$,
\beq\label{dbar}~~~~
\dbar_2^{l-j}\dbar_3^j K(\bz,\bw) = (\b)_{l-j}(\g)_{j}z_2^{l-j}z_3^j (1-z_1\ov{w}_1)^{-\a}(1-z_2\ov{w}_2)^{-\b-l+j}(1-z_3\ov{w}_3)^{-\g-j}
\eeq
where $(x)_k$ is the Pochhammer symbol defined by $(x)_k=x(x+1)\cdots(x+k-1)$.
\end{lem}

\begin{proof}
Proof follows by computing the required derivatives of $K$ using Leibnitz rule.
\end{proof}

\begin{lem}\label{ddbar}
Let $K$ be the reproducing kernel on $\D^3$ defined by 
$$K(\bz,\bw) = (1-z_1\ov{w}_1)^{-\a}(1-z_2\ov{w}_2)^{-\b}(1-z_3\ov{w}_3)^{-\g}, \text{ }, \bz,\bw \in \D^3.$$
Then, for any $l,j \in \N \cup \{0\}$ with $0 \leq j \leq l$, $k \leq l$ and $\bz,\bw \in \D^3$,
\beq\label{ddbar1}
\nonumber\del_2^k\dbar_2^{l-j}\dbar_3^j K(\bz,\bw) &=& \left[\sum_{t=0}^k{k \choose t}\frac{(l-j)!}{(l-j-k+t)!}\frac{(\b)_k}{(\b)_{k-t}}(z_2\ov{w}_2)^t\right](\b)_{l-j}(\g)_{j}z_2^{l-j-k}z_3^j\times\\
&& (1-z_1\ov{w}_1)^{-\a}(1-z_2\ov{w}_2)^{-\b-(l-j+k)}(1-z_3\ov{w}_3)^{-\g-j}.
\eeq
\end{lem}

\begin{proof}
We begin by pointing out from Lemma $\ref{lem1}$ that the desired identity holds for $k=0$, namely,
$$\dbar_2^{l-j} \dbar_3^j K(\bz,\bw) = (\b)_{l-j}(\g)_jz_2^{l-j}z_3^j(1-z_1\ov{w}_1)^{-\a}(1-z_2\ov{w}_2)^{-\b-l+j}(1-z_3\ov{w}_3)^{-\g-j},$$
for $0 \leq j \leq l$. Differentiating both sides of this equation \w $z_2$, we have that
\Bea
\del_2\dbar_2^{l-j}\dbar_3^j K(\bz,\bw) 
%&=& (\b)_{l-j}(\b+l-j)(\g)_jz_2^{l-j}\ov{w}_2z_3^j(1-z_1\ov{w}_1)^{-\a}(1-z_2\ov{w}_2)^{-\b-l+j-1}(1-z_3\ov{w}_3)^{-\g-j}\\
%&+& (\b)_{l-j}(\g)_j(l-j)z_2^{l-j-1}z_3^j(1-z_1\ov{w}_1)^{-\a}(1-z_2\ov{w}_2)^{-\b-l+j}(1-z_3\ov{w}_3)^{-\g-j}\\
%&=& [(\b+l-j)z_2\ov{w}_2+(l-j)(1-z_2\ov{w}_2)](\b)_{l-j}(\g)_jz_2^{l-j-1}z_3^j\times\\
%&&(1-z_1\ov{w}_1)^{-\a}(1-z_2\ov{w}_2)^{-\b-l+j-1}(1-z_3\ov{w}_3)^{-\g-j}\\
&=& [(l-j)+\b z_2\ov{w}_2](\b)_{l-j}(\g)_jz_2^{l-j-1}z_3^j\\
&& \times (1-z_1\ov{w}_1)^{-\a}(1-z_2\ov{w}_2)^{-\b-l+j-1}(1-z_3\ov{w}_3)^{-\g-j}.
\Eea
verifying the identity for $k=1$. We now assume that $\eqref{ddbar1}$ holds true for $k-1$. Then we have from $\eqref{ddbar1}$ with $k-1$ instead of $k$,
\Bea
\del_2^{k-1}\dbar_2^{l-j}\dbar_3^j K(\bz,\bw) &=& \left[\sum_{t=0}^{k-1}{k-1 \choose t}\frac{(l-j)!}{(l-j-k+t+1)!}\frac{(\b)_{k-1}}{(\b)_{k-1-t}}(z_2\ov{w}_2)^t\right](\b)_{l-j}(\g)_{j}\times\\
&& z_2^{l-j-k+1}z_3^j (1-z_1\ov{w}_1)^{-\a}(1-z_2\ov{w}_2)^{-\b-(l-j+k-1)}(1-z_3\ov{w}_3)^{-\g-j}.
\Eea
Differentiating both sides \w $z_2$ we get
\Bea
\del_2^k\dbar_2^{l-j}\dbar_3^j K(\bz,\bw) &=& \left[\sum_{t=0}^{k-1}{k-1 \choose t}\frac{(l-j)!}{(l-j-k+t+1)!}\frac{(\b)_{k-1}}{(\b)_{k-1-t}}t z_2^{t-1}\ov{w}_2^t\right](\b)_{l-j}(\g)_{j}\times\\
&& z_2^{l-j-k+1}z_3^j (1-z_1\ov{w}_1)^{-\a}(1-z_2\ov{w}_2)^{-\b-(l-j+k-1)}(1-z_3\ov{w}_3)^{-\g-j}\\
&+& \left[\sum_{t=0}^{k-1}{k-1 \choose t}\frac{(l-j)!}{(l-j-k+t+1)!}\frac{(\b)_{k-1}}{(\b)_{k-1-t}}(z_2\ov{w}_2)^t\right]\times\\
&&\{(l-j-k+1)+(2k-2+\b)z_2\ov{w}_2\}(\b)_{l-j}(\g)_{j}\times\\
&& z_2^{l-j-k}z_3^j (1-z_1\ov{w}_1)^{-\a}(1-z_2\ov{w}_2)^{-\b-(l-j+k)}(1-z_3\ov{w}_3)^{-\g-j}\\
&=& \left[\sum_{t=0}^{k-1}{k-1 \choose t}\frac{(l-j)!}{(l-j-k+t+1)!}\frac{(\b)_{k-1}}{(\b)_{k-1-t}}(z_2\ov{w}_2)^t\right]\times\\
&&\{(l-j-k+1+t)+(2k-2+\b-t)z_2\ov{w}_2\}(\b)_{l-j}(\g)_{j}\times\\
&& z_2^{l-j-k}z_3^j (1-z_1\ov{w}_1)^{-\a}(1-z_2\ov{w}_2)^{-\b-(l-j+k)}(1-z_3\ov{w}_3)^{-\g-j}.
\Eea

So it is enough to show that 
\Bea
\left[\sum_{t=0}^k{k \choose t}\frac{(l-j)!}{(l-j-k+t)!}\frac{(\b)_k}{(\b)_{k-t}}(z_2\ov{w}_2)^t\right] \!\!&=&\!\! \left[\sum_{t=0}^{k-1}{k-1 \choose t}\frac{(l-j)!}{(l-j-k+t+1)!}\frac{(\b)_{k-1}}{(\b)_{k-1-t}}(z_2\ov{w}_2)^t\right]\\
&& \times\{(l-j-k+t+1)+(2k-2+\b-t)z_2\ov{w}_2\}.
\Eea
Let $C_t$ be the coefficient of $(z_2\ov{w}_2)^t$ in the right hand side of the above equation for $0\leq t\leq k$. A simple calculation yields that
\begin{align*}
C_t &= \frac{(l-j)!}{(l-j-k+t)!}\left[{k-1 \choose t}\frac{(\b)_{k-1}}{(\b)_{k-1-t}}+{k-1 \choose t-1}\frac{(\b)_{k-1}}{(\b)_{k-t}}\{(\b+k-t-1)+k\}\right]\\
&= \frac{(l-j)!}{(l-j-k+t)!}\frac{(\b)_{k-1}}{(\b)_{k-t}}\left[(\b+k-t-1)\left\{{k-1 \choose t}+{k-1 \choose t-1}\right\}+{k-1 \choose t-1}k\right]\\
&= \frac{(l-j)!}{(l-j-k+t)!}\frac{(\b)_{k-1}}{(\b)_{k-t}}\left[(\b+k-t-1){k \choose t}+{k-1 \choose t-1}k\right]\\
&= \frac{(l-j)!}{(l-j-k+t)!}\frac{(\b)_{k-1}}{(\b)_{k-t}}\left[(\b+k-t-1){k \choose t}+{k \choose t}t\right]\\
&= {k \choose t}\frac{(l-j)!}{(l-j-k+t)!}\frac{(\b)_k}{(\b)_{k-t}}
\end{align*}
completing the proof.
\end{proof}

\begin{thm}\label{mthm1}
Let $K$ be the reproducing kernel on $\D^3$ defined by 
$$K(\bz,\bw) = (1-z_1\ov{w}_1)^{-\a}(1-z_2\ov{w}_2)^{-\b}(1-z_3\ov{w}_3)^{-\g}, \text{ }, \bz,\bw \in \D^3,$$ and $\Z \subset \D^3$ be the complex submanifold $\Z := \{\bz=(z_1,z_2,z_3)\in\D^3:z_2=z_3 \}$. Suppose also that $s(\bw)=K(.,\ov{\bw})$, $\bw \in \D^3$ and consider the finite set of vectors $\{\si_k: 0 \leq k \leq n-1\}$ in $\HM=\H^{(\a,\b,\g)}(\D^3)$ where the vectors $\si_k$ are defined as follows
\beq\label{spv}\si_k(\bw) &:=& \sum_{j=0}^k(-1)^j {k \choose j} \frac{(\b)_k(\g)_k}{(\b)_{k-j}(\g)_j}\del_2^{k-j}\del_3^j s(\bw), \,\,\,\bw \in \D^3.\eeq
Then, for $0 \leq l,k \leq n-1$ with $l\neq k$ and $\bw \in \Z$,
\begin{enumerate}
\item[(i)] $\norm{\si_k(\bw)}^2 = k!\sum_{i=0}^k(-1)^i {k \choose i} \frac{(\b)^2_k(\g)^2_k}{(\b)_{k-i}(\g)_i}(1-|w_1|^2)^{-\a}(1-|w_2|^2)^{-\b-\g-2k},$
\item[(ii)] $\langle\si_k(\bw),\si_l(\bw)\rangle=0.$
\end{enumerate}
\end{thm}

\begin{proof}
We first compute the norm of the vectors $\si_k(\bw)$, $0 \leq k \leq n-1$ for $\bw\in\Z$. Note, from the equation $\eqref{ddbar1}$, that 
\Bea
\del_2^{k-i}\langle s(\bw),\si_k(\bw)\rangle &=& \sum_{j=0}^k(-1)^j{k \choose j}(\b)_k(\g)_k\left[\sum_{t=0}^{k-i}{k-i \choose t}\frac{(k-j)!(\b)_{k-i}}{(i-j+t)!(\b)_{k-i-t}}|w_2|^{2t}\right]\times\\
&& w_2^{i-j} w_3^j(1-|w_1|^2)^{-\a}(1-|w_2|^2)^{-\b-(2k-i-j)}(1-|w_3|^2)^{-\g-j},
\Eea
for $0 \leq i,j \leq k$ and $\bw \in \D^3$. Therefore, a similar calculation as in Lemma $\ref{ddbar}$ yields, for $\bw \in \Z$, that
\begin{align*}
&\del_2^{k-i}\del_3^i\langle s(\bw),\si_k(\bw)\rangle\\
&= \sum_{j=0}^k(-1)^j{k \choose j}(\b)_k(\g)_k\left[\sum_{t=0}^{k-i}{k-i \choose t}\frac{(k-j)!(\b)_{k-i}}{(i-j+t)!(\b)_{k-i-t}}|w_2|^{2t}\right]\times\\
& \left[\sum_{l=0}^i{i \choose l}\frac{j!(\g)_i}{(j-i+l)!(\g)_{i-l}}|w_2|^{2l}\right](1-|w_1|^2)^{-\a}(1-|w_2|^2)^{-\b-\g-2k}\\
&= \left(\sum_{p=0}^kC_p|w_2|^{2p}\right)(\b)_k(\g)_k(1-|w_1|^2)^{-\a}(1-|w_2|^2)^{-\b-\g-2k}
\end{align*}
where $C_p$ is the coefficient of $|w_2|^{2p}$ in the last equation. We note that $C_0=(-1)^ik!$ and, for $p \geq 1$,
\begin{align*}
C_p &= \sum_{j=0}^k(-1)^j{k \choose j}\left[\sum_{q=0}^p{k-i \choose q}{i \choose p-q}\frac{(k-j)!j!(\b)_{k-i}(\g)_i}{(i-j+q)!(j-i+p-q)!(\b)_{k-i-q}(\g)_{i-p+q}}\right]\\
&= \sum_{q=0}^p{k-i \choose q}{i \choose p-q}\frac{(\b)_{k-i}(\g)_i}{(\b)_{k-i-q}(\g)_{i-p+q}}\left[ \sum_{j=0}^k(-1)^j{k \choose j}\frac{(k-j)!j!}{(i-j+q)!(j-i+p-q)!}\right]\\
&= 0
\end{align*}
where the last equality holds since
\Bea
\sum_{j=0}^k(-1)^j{k \choose j}\frac{(k-j)!j!}{(i-j+q)!(j-i+p-q)!} &=& k!\sum_{j=i-p+q}^{i+q}(-1)^j\frac{1}{(i-j+q)!(j-i+p-q)!}\\
&=& k!\sum_{\nu=0}^p(-1)^{\nu+i-p+q}\frac{1}{\nu!(p-\nu)!}\\
&=& (-1)^{i-p+q}\frac{k!}{p!}\sum_{\nu=0}^p(-1)^{\nu}\frac{p!}{\nu!(p-\nu)!}\\
&=& 0
\Eea
as $p>0$. Thus, we have, for $0 \leq i \leq k$ and $\bw \in \Z$, that 
$$\del_2^{k-i}\del_3^i\langle s(\bw),\si_k(\bw)\rangle = (-1)^ik!(\b)_k(\g)_k(1-|w_1|^2)^{-\a}(1-|w_2|^2)^{-\b-\g-2k}$$ and hence, for $\bw \in \Z$ and $0 \leq k \leq n$, it follows from the definition of $\si_k(\bw)$ that 
$$\norm{\si_k(\bw)}^2 = k!\sum_{i=0}^k(-1)^i {k \choose i} \frac{(\b)^2_k(\g)^2_k}{(\b)_{k-i}(\g)_i}(1-|w_1|^2)^{-\a}(1-|w_2|^2)^{-\b-\g-2k}.$$
This completes the proof of (i).

In order to prove (ii), we let $0 \leq l,k \leq n-1$ and $\bw \in \Z$, that is, $\bw=(w_1,w_2,w_3)$ with $w_2=w_3$. Since
$$\langle\si_k(\bw),\si_l(\bw)\rangle=\ov{\langle\si_l(\bw),\si_k(\bw)\rangle}$$ it is enough to prove the desired identity for $0 \leq k < l \leq n-1$. We note, from the definition of $\si_k(\bw)$, that 
$$\langle\si_k(\bw),\si_l(\bw)\rangle=\sum_{j=0}^k(-1)^j {k \choose j} \frac{(\b)_k(\g)_k}{(\b)_{k-j}(\g)_j}\del_2^{k-j}\del_3^j \langle s(\bw),\si_l(\bw)\rangle.$$
Thus, it is again enough to show, for $0 \leq k < l \leq n-1$ and $\bw \in \Z$, that 
\beq\label{eq1}\del_2^{k-j}\del_3^j \langle s(\bw),\si_l(\bw)\rangle=0,\,\,\,0 \leq j \leq k.\eeq
Therefore, we first prove that $\eqref{eq1}$ holds for any $k$ with $0 \leq k < l \leq n-1$ and $j=0$. %In other words, we wish to show that 
%\beq\label{eq2}\del_2^k \langle s(\bw),\si_l(\bw)\rangle=0, \text{ for }0 \leq k < l \leq n,\bw \in \Z.\eeq
Now from Lemma $\ref{ddbar}$, for $0 \leq k < l \leq n-1$ and $\bw \in \D^3$, we have that 
\begin{align*}
&\del_2^k \langle s(\bw),\si_l(\bw)\rangle\\
&= \sum_{j=0}^l(-1)^j{l \choose j}\frac{(\b)_l(\g)_l}{(\b)_{l-j}(\g)_j}\left[\sum_{t=0}^k{k \choose t}\frac{(l-j)!}{(l-j-k+t)!}\frac{(\b)_k}{(\b)_{k-t}}|w_2|^{2t}\right]\times\\
& (\b)_{l-j}(\g)_{j} w_2^{l-j-k} w_3^j (1-|w_1|^2)^{-\a}(1-|w_2|^2)^{-\b-(l-j+k)}(1-|w_3|^2)^{-\g-j}.
\end{align*}
Consequently, restricting the equation above to $\Z \subset \D^3$ and writing $\bw=(w_1,w_2,w_2) \in \Z$, we have
\begin{align*}
\del_2^k \langle s(\bw),\si_l(\bw)\rangle &= (\b)_l(\g)_l\left[\sum_{t=0}^k{k \choose t}\frac{(\b)_k}{(\b)_{k-t}}\left(\sum_{j=0}^l(-1)^j{l \choose j}\frac{(l-j)!}{(l-j-k+t)!}\right)|w_2|^{2t}\right]\times\\
& w_2^{l-k}(1-|w_1|^2)^{-\a}(1-|w_2|^2)^{-\b-\g-(l+k)}
\end{align*}
Note that, for $0 \leq t \leq k$,
\Bea
\sum_{j=0}^l(-1)^j{l \choose j}\frac{(l-j)!}{(l-j-k+t)!} &=& l! \sum_{j=0}^{l-(k-t)}(-1)^j\frac{1}{j!(l-j-(k-t))!}\\
&=& \frac{l!}{(l-(k-t))!}\sum_{j=0}^{l-(k-t)}(-1)^j\frac{(l-(k-t))!}{j!(l-j-(k-t))!}\\
&=& \frac{l!}{(l-(k-t))!}(1-1)^{l-(k-t)}\\
&=& 0
\Eea
where the last equality holds as $l-k+t > 0$. Thus, it verifies our claim. It then remains to show that $\eqref{eq1}$ holds for $0 < j \leq k$ which follows from the following lemma.
\end{proof}

\begin{lem}
Suppose that $f:\D^3 \ra \C$ is a smooth function such that, for $\bw \in \Z=\{\bz=(z_1,z_2,z_3)\in\D^3:z_2=z_3\}$,
\beq\label{di}\del_2^i f(\bw) &=& 0,\,\,\,0 \leq i \leq k. \eeq
Then, for $0 \leq i \leq k$ and $\bw \in \Z$, $$\del_2^{k-i}\del_3^i f(\bw) =0.$$
\end{lem}

\begin{proof}
We begin with the observation that, for $p \in \Z$, $p+(0,z_2,z_2) \in \Z$ with $|z_2|<\delta$, for some suitable $\delta >0$. Consequently, we have, for $p \in \Z$, that 
$$\del_2|_p + \del_3|_p := {\frac{\del}{\del z_2}}|_p +{\frac{\del}{\del z_3}}|_p \in T_p\Z$$
where by $T_p\Z$ we mean the holomorphic tangent space to $\Z$ at $p$ as a subspace of $T_p \C^3 := \text{span}_{\C}\{\frac{\del}{\del z_1},\frac{\del}{\del z_2},\frac{\del}{\del z_3}\}$.

We now prove the desired identity with the help of mathematical induction on $k$. Let us start with $k=1$. In this case, our hypothesis reduces to 
\beq\label{dieq1}f(\bw) = \del_2 f(\bw) = 0, \,\,\,\bw \in \Z,\eeq
and we are supposed to show that $\del_3 f(\bw)=0$ on $\Z$. Since $f$ is identically zero on $\Z$ and $\del_2|_p + \del_3|_p \in T_p\Z$, $p \in \Z$, it is evident that
$$(\del_2|_p + \del_3|_p)f =0.$$
Then from $\eqref{dieq1}$ it follows that $\del_3 f(\bw) =0$ for all $\bw \in \Z$. Assuming 
\beq\label{dieq2} \del_2^{k-1-j}\del_3^j f(\bw) =0, \,\,\,0 \leq j \leq k-1 \text{ and } \bw \in \Z,\eeq
we now show that 
$$\del_2^{k-j}\del_3^j f(\bw) =0, \,\,\,0 \leq j \leq k, \,\,\,\bw \in \Z.$$ 

Again, we use mathematical induction at this stage on $j$, $0 \leq j \leq k$, to show the above identity. Note that, for $j=0$, the identity above becomes $\del_2^k f(\bw)=0$, $\bw \in \Z$, which is given in the hypothesis $\eqref{di}$ with $i=k$. So let us consider $j=1$ and in this case, we show that $\del_2^{k-1}\del_3 f(\bw) =0$, for $\bw \in \Z$. From the equation $\eqref{di}$ it follows that $$\del_2^{k-1} f(\bw) = \del_2^k f(\bw) =0,$$ on $\Z$ and $\del_2|_p + \del_3|_p \in T_p\Z$, $p \in \Z$. Therefore, as before, a similar computation verifies the claim.

Now assume that 
\beq\label{dieq3}\del_2^{k-j+1}\del_3^{j-1} f(\bw) =0 \text{ on } \Z.\eeq
Since $0 \leq j \leq k$ and $j \geq 2$ we observe, from $\eqref{dieq2}$, that $$\del_2^{k-1-j+1}\del_3^{j-1} f(\bw) = \del_2^{k-j}\del_3^{j-1} f(\bw) =0$$ on $\Z$. Hence we have 
$$(\del_2|_p + \del_3|_p)(\del_2^{k-j}\del_3^{j-1} f)= \del_2^{k-j+1}\del_3^{j-1} f(p) +\del_2^{k-j}\del_3^{j} f(p)=0,$$ for $p \in \Z$, which together with $\eqref{dieq3}$ complete the proof.
\end{proof}

Observe that the submanifold $\Z=\{\bz=(z_1,z_2,z_3)\in\D^3:z_2=z_3\}$  introduced in Theorem $\ref{mthm1}$, is biholomorphic to $\D^2$ via the biholomorphism $\phi: \D^2 \ra \Z$ defined by $\phi(z_1,z_2)=(z_1,z_2,z_2)$. In other words, $(\Z,\phi^{-1})$ is a global holomorphic co-ordinate chart for $\Z$. Also, $\phi(\{(z_1,z_2)\in \D^2:z_1=z_2\})=\Delta$, the diagonal submanifold of $\Z$, and, for $p \in \D^2$, $$D\phi|_p\left(\frac{\del}{\del z_2}\vert_p\right)={\frac{\del}{\del z_2}}|_{\phi(p)} +{\frac{\del}{\del z_3}}|_{\phi(p)}.$$
Thus, ${\frac{\del}{\del z_2}}|_{\phi(p)} +{\frac{\del}{\del z_3}}|_{\phi(p)}$ is a complementary direction to $\Delta$ in $\Z$. We denote this complementary direction by $\phi_*(\del_2)$.

For $0 \leq l \leq n-1$, let $K_l:\Z \times \Z \ra \C$ be the reproducing kernel 
\beq\label{rkjsb} K_l(\textbf{z},\textbf{w}):= l!\sum_{i=0}^l(-1)^i {l \choose i} \frac{(\b)^2_l(\g)^2_l}{(\b)_{l-i}(\g)_i}(1-z_1\ov{w}_1)^{-\a}(1-z_2\ov{w}_2)^{-\b-\g-2l},\eeq
for $\textbf{z}=(z_1,z_2,z_2), \textbf{w}=(w_1,w_2,w_2)\in \Z$. Also assume that $\HM_l$ is the reproducing kernel Hilbert space corresponding to $K_l$ and $M^{(l)}=(M_1^{(l)},M_2^{(l)})$ is the pair of multiplication operators by the co-ordinate functions $z_1$ and $z_2$. We should also point out that we are considering the submanifold $\Z$ as $\D^2$ under the identification $\Phi:\Z \ra \D^2$ defined by $\textbf{z} \mapsto (z_1,z_2)$. Then note that $(M^{(l)})^*$ is in $\mathrm B_1(\Z)$ since $K_l$ is equivalent to the weighted Bergman kernel $(1-z_1\ov{w}_1)^{-\a}(1-z_2\ov{w}_2)^{-\b-\g-2l}$ for each $l=0,\hdots,n-1$. Let $E_l$, for $0 \leq l \leq n-1$, be the hermitian holomorphic line bundle associated to $(M^{(l)})^*$, respectively, over $\Z$. Theorem $\ref{mthm1}$ then translates to the following fact.

\begin{cor}
The {\h}s $E_l\ra \Z$, for $0 \leq l \leq n-1$, are mutually orthogonal to each other as hermitian holomorphic subbundles of  the trivial bundle $\HM\times\D^3$ restricted to $\Z$.
\end{cor}  

Now, for each $l$ with $0 \leq l \leq n$, following the description given in the beginning of the present section, we can construct the jet bundles $J^{(r)}E_l|_{\Delta}$ over $\Delta$ relative to $\Delta \subset \Z$ by declaring 
$$J^{(r)}(\si_l):=\{\si_l,\phi_*(\del_2)\si_l,\hdots,\phi_*(\del_2)^r\si_l\}$$ as a global holomorphic frame over $\Delta$. Then $J^{(r)}E_l|_{\Delta} \ra \Delta$ becomes a {\h} over $\Delta$ with the hermitian structure given by
$$\langle\phi_*(\del_2)^{r_1}\si_l,\phi_*(\del_2)^{r_2}\si_l\rangle=(\del_2+\del_3)^{r_1}(\dbar_2+\dbar_3)^{r_2}\norm{\si_l}^2,$$ for $0 \leq r_1,r_2 \leq r$, on $\Delta$. Consequently, with the help of Theorem $\ref{mthm1}$ we get, for $0 \leq l,k \leq n$ with $l\neq k$ and $0 \leq r_1,r_2 \leq r$, 
$$\langle\phi_*(\del_2)^{r_1}\si_l,\phi_*(\del_2)^{r_2}\si_k\rangle =0 \text{ on } \Delta.$$

\smallskip

For $0 \leq l \leq n$ and $r \in \N$ with $r+l+1 \leq n$, $J^{(r)}E_l \ra \Delta$ is a holomorphic subbundle of $J^{(n)}E|_{\Delta} \ra \Delta$ where $J^{(n)}E|_{\Delta} \ra \Delta$ is the {\h} corresponding to the module $J^{(n)}(\HM)|_{\text{res}\Delta}$ introduced earlier in the present section. We also point out that, for $\bw \in \Delta$, $$J^{(n)}E|_{\bw}=J^{(n)}E_0|_{\bw} + J^{(n-1)}E_1|_{\bw} + \cdots+J^{(1)}E_{n-1}|_{\bw} + E_n|_{\bw}.$$ Indeed, the right hand side of the above equation is a linear subspace of $J^{(n)}E|_{\bw}$ and $$\dim_{\C}(J^{(n)}E_0|_{\bw}+\cdots+E_n|_{\bw})=\dim_{\C}(J^{(n)}E|_{\bw}).$$ Thus we have shown that $$J^{(n)}E|_{\Delta} = \bigoplus_{l=0}^nJ^{(n-l)}E_l|_{\Delta}.$$ Furthermore, it follows from Theorem 4.2 in \cite{OICHO} that the compression of the multiplication operator on $J^{(n)}(\HM)$ onto the qutient module $J^{(n-l)}(\HM_l)|_{\text{res}\Delta}$ is irreducible which in turn implies that the jet bundles $J^{(n-l)}E_l|_{\Delta} \ra \Delta$ are irreducible where $\HM_l$ is the Hilbert module over $\A(\Z)$ with reproducing kernel $K_l$ on $\Z$ introduced in \eqref{rkjsb}. Thus, we have the following result.

\begin{thm}\label{mthm2}
Let $E \ra \D^3$ be the hermitian holomorphic line bundle associated to the Hilbert module $\HM=\H^{(\a,\b,\g)}(\D^3)$ and $\HM^{(n)}$ be the submodule of holomorphic functions in $\HM$ vanishing of order $n$ along the diagonal set $\Delta \subset \D^3$. Let $J^{(n)}E|_{\Delta}$ be the jet bundle of order $n$ relative to $\Delta$ associated to the Hilbert module $J^{(n)}(\HM)|_{\text{res}\Delta}$. Then $$J^{(n)}E|_{\Delta} = \bigoplus_{l=0}^{n-1}J^{(n-l)}E_l|_{\Delta}$$ where $J^{(n-l)}E_l|_{\Delta}$ is the jet bundle as above for $0\leq l\leq n-1$. Moreover, the jet bundles $J^{(n-l)}E_l|_{\Delta}\ra \Delta$, for $0\leq l\leq n-1$, are irreducible as \h. 
\end{thm}

Recall following well-known lemma which essentially relates the reducibility of an element in the Cowen-Douglas class $\mathrm B_r(\O)$ to that of the {\h} associated to it.

\begin{lem}\label{bundle vs operator}
Let $\H_K$ be a {\rkhs} with the reproducing kernel $K$ on $\O^*$ such that the adjoint of the tuple of multiplication operators $\M=(M_{z_1},\hdots,M_{z_m})$ is in $\mathrm B_r(\O)$. Suppose that the {\h} $E_{\M^*}$ associated to $\M^*$ is unitarily equivalent to $E_1\oplus E_2$ as {\h} over $\O$. We also assume that both $E_1$ and $E_2$ are {\h}s associated to the adjoint of the multiplication operators on some reproducing kernel Hilbert spaces $\H_{K_1}$ and $\H_{K_2}$ with reproducing kernels $K_1$ and $K_2$ on $\O^*$, respectively. Then $\M^*$ is reducible.
\end{lem}

\begin{proof}
Let $\f=\{s_i(w)=K(.,\ov{w})\si_i:1\leq i\leq r, w \in \O\}$ be a holomorphic frame for $E_{\M^*} \ra \O$ and $\Phi: E_{\M^*}\ra E_1\oplus E_2$ be the isomorphism covering the identity mapping on the base. Let $((\Phi_{ij}))_{i,j=1}^r$ be the matrix of $\Phi$. We then have that $$((\Phi_{ij}(w)))_{i,j=1}^rK(\ov{w},\ov{w}){((\Phi_{ij}(w)))_{i,j=1}^r}^*=\begin{pmatrix}
K_1(\ov{w},\ov{w}) & 0\\
0 & K_2(\ov{w},\ov{w})
\end{pmatrix}=\tilde{K}(\ov{w},\ov{w}).$$ Thus it shows that the map $\G:\H_{K}\ra\H_{\tilde{K}}$ defined by $$\G(K(.,\ov{w})\si_j)=\sum_{i=1}^r \Phi_{ij}(w)\tilde{K}(.,\ov{w})\si_i$$ is unitary. Moreover, it is seen that $$\G M_{z_l}^*(K(.,\ov{w})\si_j)=\widetilde{M}_{z_l}^*\G(K(.,\ov{w})\si_j)~~\text{for }1\leq l \leq m~\text{and }1\leq j\leq r$$ which is same as $$M_{z_l}^*=\G^*\widetilde{M}_{z_l}^*\G,~~~\text{for }1\leq l\leq m.$$ Let $V_1=\G^{-1}(\H_{K_1}\oplus \{0\})$ and $V_2=\G^{-1}(\{0\}\oplus \H_{K_2})$. Then the identity above shows that $$M_{z_l}^*=M_{z_l}^*|_{V_1}+M_{z_l}^*|_{V_2},~\text{for }1\leq l \leq m.$$ Thus, $\M^*$ is reducible.
\end{proof}

Now we prove an analogous statement as the previous theorem for general $m$.

\begin{thm}\label{mthm3}
Let $E \ra \D^m$ be the hermitian holomorphic line bundle associated to the Hilbert module $\HM=\H^{(\a)}(\D^m)$ with $\a\in\R_{>0}^m$ and $\HM^{(n)}$ be the submodule of holomorphic functions in $\HM$ vanishing of order $n$ along the diagonal set $\Delta \subset \D^m$. Let $J^{(n)}E|_{\Delta}$ be the jet bundle of order $n$ relative to $\Delta$ associated to the Hilbert module $J^{(n)}(\HM)|_{\text{res}\Delta}$. Then 
\beq\label{bjb} J^{(n)}E|_{\Delta}=\bigoplus_{l_{m-2}=0}^{n}\left(\bigoplus_{l_{m-3}=0}^{n-l_{m-2}}\cdots\left(\bigoplus_{l_1=0}^{n-l_{m-2}-\cdots-l_2}J^{(n-l_{m-2}-\cdots-l_2-l_1)}E_{l_1}|_{\Delta}\right)\right)\eeq wehre, for $0\leq l_1\leq n-l_{m-2}-\cdots-l_2$, $E_{l_1}$ is the {\hl} corresponding to the Hilbert module $\HM_{l_1}$ with the reproducing kernel
$$K_{l_1}(\bz,\bw):=C_{l_{m-2}}(1-z_1\ov{w}_1)^{\a_1}(1-z_2\ov{w}_2)^{-\a_2-\hdots-\a_m-2(l_{m-2}+\hdots+l_1)},~~~\bz,\bw\in\Z,$$
for some positive constant $C_{l_{m-2}}$ and $\Z=\{\bz\in\D^m:z_2=\cdots=z_m\}$. Moreover, 
$$C_1=(l_{m-2})!\sum_{i_1=0}^{l_{m-2}}(-1)^{i_1}{l_{m-2}\choose i_1}\dfrac{(\a_{m-1})^2_{l_{m-2}}(\a_m)_{l_{m-2}}^2}{(\a_{m-1})_{l_{m-2}-i_1}(\a_m)_{i_1}},$$ and for $2\leq j\leq m-2$, 
$$\dfrac{C_j}{C_{j-1}}=(l_{m-j-1})!\sum_{i_j=0}^{l_{m-j-1}}(-1)^{i_j}{l_{m-j-1}\choose i_j}A_j$$ where 
$$A_j=\dfrac{(\a_{m-j})^2_{l_{m-j-1}}(\a_{m-j+1}+\hdots+\a_m+2(l_{m-2}+\hdots+l_{m-j}))_{l_{m-j-1}-i_j}}{(\a_{m-j})^2_{l_{m-j-1}}(\a_{m-j+1}+\hdots+\a_m+2(l_{m-2}+\hdots+l_{m-j}))_{i_j}}.$$  Moreover, the jet bundles $J^{(n-l_{m-2}-\cdots-l_2-l_1)}E_{l_1}|_{\Delta}\ra \Delta$, for $0\leq l_{m-2}\leq l_{m-2}+l_{m-3}\leq \cdots\leq l_{m-2}+\cdots+l_2+l_1\leq n$, are irreducible as \h. 
\end{thm}

\begin{proof}
Let $E \ra \D^m$ be the hermitian holomorphic line bundle as given with the hermitian metric obtained from the reproducing kernel 
$$K(\bz,\bw):= K^{(\a)}(\bz,\bw) = \prod_{j=1}^m(1-z_j\ov{w}_j)^{-\a_j},~~~\bz,\bw \in \D^m$$ of the Hilbert module $\HM=\H^{(\a)}(\D^m)$ with $\a\in\R_{>0}^m$. Then we prove the desired equality with the help of mathematical induction on $m$.

We note that the base case of this induction follows from Theorem \ref{mthm2}. So we assume that the equation \eqref{bjb} holds true for $0\leq j\leq m-1$. Let $Z$ be the submanifold of $\D^m$ defined as $Z=\{\bz\in\D^m:z_{m-1}=z_m\}$. Then $Z$ is biholomorphic to $\D^{m-1}$ via the biholomorphism $\phi: \D^{m-1} \ra Z$ defined by $\phi(z_1,\hdots,z_m)=(z_1,\hdots,z_{m-2},z_{m-1},z_{m-1})$. In particular, $(Z,\phi^{-1})$ is a global holomorphic co-ordinate chart for $Z$. Furthermore, $\phi(\{(z_1,\hdots,z_{m-1})\in \D^{m-1}:z_1=\hdots=z_{m-1}\})=\Delta$, the diagonal submanifold of $Z$, and,for $p \in \D^{m-1}$, $$D\phi|_p\left(\frac{\del}{\del z_{m-1}}|_p\right)={\frac{\del}{\del z_{m-1}}}|_{\phi(p)} +{\frac{\del}{\del z_m}}|_{\phi(p)}.$$
Thus, ${\frac{\del}{\del z_{m-1}}}|_{\phi(p)} +{\frac{\del}{\del z_m}}|_{\phi(p)}$ is a complementary direction to $\Delta$ in $Z$ (which is also referred as normal direction in Section 5 in \cite{OUIOQM}). We denote this complementary direction by $\phi_*(\del_2)$.

For $0 \leq l_{m-2} \leq n$, let $\HM_{l_{m-2}}$ be the reproducing kernel Hilbert space with $K_{l_{m-2}}:Z \times Z \ra \C$ be the reproducing kernel 
$$ K_{l_{m-2}}(\bz,\bw):= C_1\prod_{t=1}^{m-2}(1-z_t\ov{w}_t)^{-\a_t}(1-z_{m-1}\ov{w}_{m-1})^{-\a_{m-1}-\a_m-2l_{m-2}},$$
for $\bz,\bw\in Z$. Denote $\M^{(l_{m-2})}:=(M_1^{(l_{m-2})},\hdots,M_{m-1}^{(l_{m-2})})$ as the multiplication operator by the co-ordinate functions. Here we should point out that we are considering the submanifold $Z$ as $\D^{m-1}$ under the identification $\phi:Z \ra \D^{m-1}$ defined by $\bz \mapsto (z_1,\hdots,z_{m-1})$. Observe that $\M^{(l_{m-2})}$ is in $\mathrm B_1(Z)$ since $K_{l_{m-2}}$ is equivalent to the weighted Bergman kernel $$\tilde{K}_{l_{m-2}}(\bz,\bw):=\prod_{t=1}^{m-2}(1-z_t\ov{w}_t)^{-\a_t}(1-z_{m-1}\ov{w}_{m-1})^{-\a_{m-1}-\a_m-2l_{m-2}},$$ for each $l_{m-2}=1,\hdots,n$ and $\bz,\bw\in Z$. Consequently, $M^{(l_{m-2})}$ gives rise to a hermitian holomorphic line bundle $E_{l_{m-2}}\ra Z$, for $0 \leq l_{m-2} \leq n$. Then as in the proof of Theorem $\ref{mthm1}$ we have that the {\hl}s $E_{l_{m-2}}\ra Z$, for $0 \leq l_{m-2} \leq n$, are mutually orthogonal to each other. A similar argument as in Theorem \ref{mthm2} yields that 
$$J^{(n)}E|_{\Delta}=\bigoplus_{l_{m-2}=0}^nJ^{(n-l_{m-2})}E_{l_{m-2}}|_{\Delta}$$ where $E_{l_{m-2}}\ra\Z$ is the {\hl} as defined above.

Since $\M^{(l_{m-2})}$ is in $\mathrm B_1(Z)$ and $Z$ is biholomorphic to $\D^{m-1}$ we apply the induction hypothesis to the subbundles $J^{(n-l_{m-2})}E_{l_{m-2}}|_{\Delta}\ra\Delta$ to obtain the desired equality. Moreover, the irreducibility of jet bundles $J^{(n-l_{m-2}-\cdots-l_2-l_1)}E_{l_1}|_{\Delta}\ra \Delta$, for $0\leq l_{m-2}\leq l_{m-2}+l_{m-3}\leq \cdots\leq l_{m-2}+\cdots+l_2+l_1\leq n$, follows from Theorem 4.2 in \cite{OICHO}.
\end{proof}

We now state our main result of this section which is essentially a corollary of Theorem $\ref{mthm3}$. Let, for each $0 \leq l_j \leq n-1$ with $0\leq l_{m-2}\leq l_{m-2}+l_{m-3}\leq \cdots\leq l_{m-2}+\cdots+l_2+l_1\leq n-1$ and $1\leq j\leq m-2$, $\ell=(l_1,\hdots,l_{m-2})$ and $\HM_{\ell}$ be the Hilbert module over $\A(\Z)$ with the reproducing kernel $$K_{\ell}(\bz,\bw):=C_{l_{m-2}}(1-z_1\ov{w}_1)^{\l_1}(1-z_2\ov{w}_2)^{-\l_2-\hdots-\l_m-2|\ell|},~~~\bz,\bw\in\Z,$$ for some positive constant $C_{l_{m-2}}$ and $\Z=\{\bz\in\D^m:z_2=\cdots=z_m\}$ as introduced in Theorem \ref{mthm3}. Denote $\HM^{(\ell)}_q$ as the quotient module obtained from the submodule of functions in $\HM_{\ell}$ which vanish of order $n-|\ell|$ along the diagonal set $\Delta \subset \Z$, for each $0 \leq l_j \leq n-1$ with $0\leq l_{m-2}\leq l_{m-2}+l_{m-3}\leq \cdots\leq l_{m-2}+\cdots+l_2+l_1\leq n-1$ and $1\leq j\leq m-2$. 

\begin{thm}\label{reducibility}
Let $\HM$ be the Hilbert module $\H^{(\l)}(\D^m)$ with $\a\in\R_{>0}^m$ and $\HM_q^{(n)}$ be the quotient module obtained from the submodule $\HM^{(n)}$ of holomorphic functions in $\HM$ vanishing of order $n$ along the diagonal set $\Delta \subset \D^m$. Then
$$M_q=\bigoplus_{l_{m-2}=0}^{n-1}\left(\bigoplus_{l_{m-3}=0}^{n-1-l_{m-2}}\cdots\left(\bigoplus_{l_1=0}^{n-1-l_{m-2}-\cdots-l_2}M_q^{(\ell)}\right)\right)$$
where $\ell=(l_1,\hdots,l_{m-2})$, $M_q$ is the compression of the multiplication operator $M_{z_1}$ on $\HM$ onto the quotient module $\HM^{(n)}_q$ and $M_q^{(\ell)}$ are the compression of multiplication operator $M_{z_1}$ on $\HM_{\ell}$ onto the quotient modules $(\HM_{\ell})^{(n-|\ell|)}_q$ obtained from the submodule $\HM^{(n-|\ell|)}_{\ell}$ of the module $\HM_{\ell}$, for each $0 \leq l_j \leq n-1$ with $0\leq l_{m-2}\leq l_{m-2}+l_{m-3}\leq \cdots\leq l_{m-2}+\cdots+l_2+l_1\leq n-1$ and $1\leq j\leq m-2$. Moreover, each $M_q^{(\ell)}$ are irreducible.
\end{thm}

\begin{proof}
We begin with the observation that the Hilbert modules $\HM$, $\HM^{(n)}_q$, $\HM^{(\ell)}$ as well as $\HM^{(\ell)}_q$ are all in the Cowen-Douglas class (cf. \cite[Theorem 4.10]{OUIOQM}), for each $0 \leq l_j \leq n-1$ with $0\leq l_{m-2}\leq l_{m-2}+l_{m-3}\leq \cdots\leq l_{m-2}+\cdots+l_2+l_1\leq n-1$ and $1\leq j\leq m-2$. From Lemma \ref{bundle vs operator} it is enough to show that the hermitian holomorphic vector bundle $J^{(n)}E|_{\Delta} \ra \Delta$ is reducible and the reducing factors correspond the quotient modules $\HM^{(\ell)}_q$. From Theorem $\ref{mthm3}$, we are then required to show that the{\h} $\mathscr{E}_{\ell} \ra \Delta$ associated to $\HM^{(\ell)}_q$ is unitarily equivalent to the jet bundle $J^{(n-|\ell|)}E_{\ell}|_{\Delta}\ra \Delta$. 

Let $n_1=n-|\ell|$ and note, from the definition of the bundles $J^{(n_1)}E_{\ell}|_{\Delta}\ra \Delta$, that 
$$J^{(n_1)}(\si_{\ell}):=\left\{\si_{\ell},\left(\sum_{j=2}^m\del_j\right)\si_{\ell},\hdots,\left(\sum_{j=2}^m\del_j\right)^{n_1-1}\!\!\!\!\!\!\!\si_{\ell}\right\}$$
is a global holomorphic frame for the bundle $J^{(n_1)}E_{\ell}|_{\Delta}\ra \Delta$ where $\si_{\ell}(\bw)=K_{\ell}(\cdot,\ov{\bw})$ for $\bw \in \Z$. Moreover, the hermitian structure on $J^{(n_1)}E_{\ell}|_{\Delta}\ra \Delta$, for $0\leq i,j\leq n_1-1$, is given by 
\Bea\left\langle\left(\sum_{t=2}^m\del_t\right)^i\si_{\ell},\left(\sum_{t=2}^m\del_t\right)^j\si_{\ell}\right\rangle\Big\vert_{\bw}&=&\left(\sum_{t=2}^m\del_t\right)^i\left(\sum_{t=2}^m\dbar_t\right)^jK_{\ell}(\ov{\bw},\ov{\bw}),\Eea for $\bw=(w_1,w_1,w_1) \in \Delta$.
On the other hand, the jet construction presented in Section $\ref{JC}$ gives rise to the Hilbert module $J(\HM^{(\ell)})$ where $J$ is the unitary map $J: \HM^{(\ell)} \ra J(\HM^{(\ell)})$ defined by $$h\mapsto \sum_{j=0}^{n_1-1}\left(\sum_{t=2}^m\del_t\right)^jh\otimes \varepsilon_j$$ where $\{\varepsilon_j\}_{j=0}^{n_1-1}$ is the standard ordered basis for $\C^{n_1}$. Furthermore, the quotient module $\HM^{(\ell)}_q$ is unitarily equivalent to the reproducing kernel Hilbert module $J(\HM^{(\ell)})|_{\Delta}$ with the reproducing kernel $$JK_{\ell}|_{\Delta}(\ov{\bz},\ov{\bw})_{ij}=\left(\sum_{t=2}^m\del_t\right)^i\left(\sum_{t=2}^m\dbar_t\right)^jK_{\ell}(\ov{\bz},\ov{\bw})~~~\text{for}~~~\bz,\bw\in\Delta.$$ Therefore, the vector bundle $\mathscr{E}_{\ell}$ is unitarily equivalent to the jet bundle $J^{(n_1)}E_{\ell}|_{\Delta}\ra\Delta$. This completes the proof.
\end{proof}

\section{Appendix}
\subsection{Toy example I}
Let $\a,\b,\g>0$ and $\H^{(\a,\b,\g)}(\D^3)$ be the weighted Bergman module over $\D^3$ with the reproducing kernel 
$$K(\bz,\bw)=(1-z_1\ov{w}_1)^{-\a}(1-z_2\ov{w}_2)^{-\b}(1-z_3\ov{w}_3)^{-\g}~\bz,\bw\in\D^3.$$ Let us take the submodule $\H_0^{(2)}$ consisting of functions in $\H^{(\a,\b,\g)}(\D^3)$ vanishing of order $2$ along the diagonal set $\Delta=\{(z_1,z_2,z_3)\in\D^3:z_1=z_2=z_3\}$. Denote $\H_q^{(2)}$ as the quotient space $\H^{(\a,\b,\g)}(\D^3)\ominus \H^{(2)}_0$. It then follows from \cite{OUIOQM} that $\H^{(2)}_q$ is \rkhs with the reproducing kernel $$JK(z,w)=\begin{pmatrix}
         (1-z\ov{w})^2 & \b z(1-z\ov{w}) & \g z(1-z\ov{w})\\
          \b \ov{w}(1-z\ov{w}) & \b(1+\b z\ov{w}) & \b\g z\ov{w} \\
         \g \ov{w}(1-z\ov{w})  &    \b\g z\ov{w}  & \g(1+\g z\ov{w})\end{pmatrix}\times (1-z\ov{w})^{-\a-\b-\g-2}.$$ It can be seen that $\{e^{(p)}_1,e^{(p)}_2,e^{(p)}_3:p \in \N \cup \{0\}\}$ forms a basis for the quotient module $\H^{(2)}_q $ where
$$
e^{(p)}_1 \mapsto
\left(
\begin{smallmatrix}
{-(\a+\b+\g)\choose p}^{\frac{1}{2}}z^p\\
\a\sqrt{\frac{p}{\a+\b+\g}}{-(\a+\b+\g+1)\choose (p-1)}^{\frac{1}{2}}z^{p-1}\\
\b\sqrt{\frac{p}{\a+\b+\g}}{-(\a+\b+\g+1)\choose (p-1)}^{\frac{1}{2}}z^{p-1}\\
\end{smallmatrix}
\right),
~e^{(p)}_2  \mapsto
\left(
\begin{smallmatrix}
0\\
\frac{\a\b}{\sqrt{\b(\a+\g)}}\frac{1}{\sqrt{\a+\b+\g}}{-(\a+\b+\g+2)\choose (p-1)}^{\frac{1}{2}}z^{p-1}\\
\frac{\b\g}{\sqrt{\b(\a+\g)}}\frac{1}{\sqrt{\a+\b+\g}}{-(\a+\b+\g+2)\choose (p-1)}^{\frac{1}{2}}z^{p-1}\\
\end{smallmatrix}
\right)$$
$$\text{and}~~
e^{(p)}_3 \mapsto
\left(
\begin{smallmatrix}
0\\
\sqrt{\frac{\a\g}{\a+\g}}{-(\a+\b+\g+2)\choose (p-1)}^{\frac{1}{2}}z^{p-1}\\
-\sqrt{\frac{\a\g}{\a+\g}}{-(\a+\b+\g+2)\choose (p-1)}^{\frac{1}{2}}z^{p-1}\\
\end{smallmatrix}
\right).\\
$$

From the matrix of the compression of the multiplication operator $M_{z}$ onto $\H^{(2)}_q$ \w this orthonormal basis it is seen that $M_z$ is reducible. In fact, it is equivalent to the fact that the reproducing kernel $J^{(2)}K$ is equivalent to a $3\times 3$ block-diagonal matrix valued kernel. In other words, it is enough to find a $3\times 3$ invertible matrix $X$ such that $XJ^{(2)}KX^*$ is a $3\times 3$ block-diagonal matrix valued kernel.

In this regard, we first observe that $J^{(2)}K$ can be written as follows:
$$J^{(2)}K(z,w)=\left(\begin{smallmatrix}
         (1-z\ov{w}) & 0 & 0\\
          0 & 1 & 0 \\
         0 &    0  & 1\end{smallmatrix}\right)\left(\begin{smallmatrix}
         1 & 0 & 0\\
          \b \ov{w} & 1 & 0 \\
         \g \ov{w}  &    0  & 1\end{smallmatrix}\right)\left(\begin{smallmatrix}
         1 & 0 & 0\\
          0 & \b & 0 \\
         0 &    0  & \g\end{smallmatrix}\right)\left(\begin{smallmatrix}
         1 & \b z & \g z\\
          0 & 1 & 0 \\
         0 &    0  & 1\end{smallmatrix}\right)\left(\begin{smallmatrix}
         (1-z\ov{w}) & 0 & 0\\
          0 & 1 & 0 \\
         0 &    0  & 1 \end{smallmatrix}\right)\times (1-z\ov{w})^{-\a-\b-\g-2}.$$
Since the matrix $\left(\begin{smallmatrix}
         (1-z\ov{w}) & 0 & 0\\
          0 & 1 & 0 \\
         0 &    0  & 1\end{smallmatrix}\right)$ is a block diagonal matrix it is enough to have a $3\times 3$ matrix $X$ such that $X\exp{S\ov{w}}J^{(2)}K(0,0)\exp{S^*z}X^*$ is a block diagonal matrix where $S=\left(\begin{smallmatrix}
         0 & 0 & 0\\
          \b  & 0 & 0 \\
         \g   &    0  & 0 \end{smallmatrix}\right)$. It turns out that $X=\left(\begin{smallmatrix}
         1 & 0 & 0\\
          0 & 1 & 1 \\
         0 &  -\g & \b\end{smallmatrix}\right)$ serves the purpose. More precisely, we have that 
$$XJ^{(2)}K(z,w)X^*=\left(\begin{smallmatrix}
         (1-z\ov{w})^2 & (\b+\g)z(1-z\ov{w}) & 0\\
          (\b+\g)\ov{w}(1-z\ov{w}) & (\b+\g)(1+(\b+\g)z\ov{w}) & 0 \\
         0 &    0  & \b\g(\b+\g)\end{smallmatrix}\right)\times (1-z\ov{w})^{-\a-\b-\g-2}.$$

Recall that the adjoint of the multiplication operator $M_z$ on $\H^{(2)}_q$ corresponds a {\h}, say $J^{(2)}E|_{\Delta}\ra\Delta$, of rank $3$ with the hermitian structure given by the reproducing kernel $J^{(2)}K$. Then $K(.,\ov{\bw}),\del_2K(.,\ov{\bw}),\del_3K(.,\ov{\bw})$ form a holomorphic frame for $J^{(2)}E|_{\Delta}$. In this set up, we point out that $X$ defines a isometric isomorphism $\Phi^{(2)}_X$ of the vector bundle $J^{(2)}E|_{\Delta}$ as follows:
$$\Phi^{(2)}_X(K(.,\ov{\bw}))=K(.,\ov{\bw}),~\Phi^{(2)}_X(\del_2K(.,\ov{\bw}))=(\del_2+\del_3)K(.,\ov{\bw}),~\text{and}$$
$$\Phi^{(2)}_X(\del_3K(.,\ov{\bw}))=-\g\del_2K(.,\ov{\bw})+\b\del_3K(.,\ov{\bw}),~\text{for}~\bw\in\Delta.$$ 
Further, $\Phi^{(2)}_X(\del_3K(.,\ov{\bw}))$ is perpendicular to both $\Phi^{(2)}_X(K(.,\ov{\bw}))$ and $\Phi^{(2)}_X(\del_2K(.,\ov{\bw}))$. At this point, we observe that the subbundle determined the holomorphic frame $\{K(.,\ov{\bw}),(\del_2+\del_3)K(.,\ov{\bw}):\bw\in\Delta\}$ is associated to the quotient space obtained from the submodule of functions in $\H^{(\a,\b+\g)}(\D^2)\cong \H^{(\a,\b,\g)}(\D^3)|_{\Z}:=\{f|_{\Z}:f\in\H^{(\a,\b,\g)}(\D^3)\}$ vanishing of order $2$ along the diagonal set $\Delta\subset \Z\subset\D^3$ where $\Z=\{\bz\in\D^3:z_2=z_3\}$. 

Let us now consider the quotient Hilbert space $\H^{(3)}_q$ obtained from the submodule $\H^{(3)}_0$ consisting of functions in $\H^{(\a,\b,\g)}(\D^3)$ vanishing of order $3$ along the diagonal set $\Delta$. Then as above one can see, for $$X=\begin{pmatrix}
         1 & 0 & 0 & 0 & 0 & 0\\
          0 & 1 & 1 & 0 & 0 & 0\\
         0 &  0 & 0 & 1 & 2 & 1\\
         0 & -\g & \b & 0 & 0 & 0\\
         0 & 0 & 0 & -\g & \b-\g & \b\\
         0 & 0 & 0 & \g(\g+1) & -2(\b+1)(\g+1) & \b(\b+1)\end{pmatrix},$$ that $XJ^{(3)}KX^*$ becomes a block diagonal matix valued kernel. Moreover, the corresponding bundle isomorphism $\Phi^{(3)}_X:J^{(3)}E_{\Delta}\ra J^{(3)}E|_{\Delta}$ is given by the following formulas:
\Bea
\Phi^{(3)}_X(K(.,\ov{\bw}))&=& K(.,\ov{\bw})\\
\Phi^{(3)}_X(\del_2K(.,\ov{\bw})) &=& (\del_2+\del_3)K(.,\ov{\bw})\\
\Phi^{(3)}_X(\del_2K(.,\ov{\bw})) &=& (\del_2^2 + 2\del_2\del_3 + \del_3^2)K(.,\ov{\bw})\\
\Phi^{(3)}_X(\del^2_2K(.,\ov{\bw})) &=& (-\g\del_2+\b\del_3)K(.,\ov{\bw})\\
\Phi^{(3)}_X(\del_2\del_3K(.,\ov{\bw})) &=& (-\g\del_2^2 + (\b-\g)\del_2\del_3 + \b\del_3^2)K(.,\ov{\bw})\\
\Phi^{(3)}_X(\del^2_3K(.,\ov{\bw})) &=& ((\g)_2\del_2^2 - 2(\b+1)(\g+1)\del_2\del_3 + (\b)_2)K(.,\ov{\bw}).
\Eea
Set $\si_0(\bw)=K(.,\ov{\bw})$, $\si_1(\bw)=-\g\del_2K(.,\ov{\bw})+\b\del_3K(.,\ov{\bw})$ and $$\si_2(\bw)=((\g)_2\del_2^2 - 2(\b+1)(\g+1)\del_2\del_3 + (\b)_2)K(.,\ov{\bw})=\sum_{i=0}^2(-1)^i{2\choose i}\frac{(\b)_2(\g)_2}{(\b)_{2-i}(\g)_i}\del_2^{2-i}\del_3^iK(.,\ov{\bw}).$$
It then implies that the bundle isomorphism $\Phi_X^{(3)}$ breaks the jet bundle $J^{(3)}E|_{\Delta}$ into orthogonal direct sum of three holomorphic subbundles $E_0,E_1$ and $E_2$ with global holomorphic frames $\{\si_0(\bw),(\del_2+\del_3)\si_0(\bw),(\del_2+\del_3)^2\si_0(\bw)\}$, $\{\si_1(\bw),(\del_2+\del_3)\si_1(\bw)\}$ and $\{\si_2(\bw)\}$, respectively. We note that $E_0$ is the jet bundle associated to the quotient Hilbert space obtained from the submodule of functions in $\H^{(\a,\b+\g)}(\D^2)$ vanishing of order $3$ along the diagonal set $\Delta\subset \Z\subset\D^3$, $E_1$ is the jet bundle associated to the quotient Hilbert space obtained from the submodule of functions in $\H_K$ vanishing of order $2$ along the diagonal set $\Delta\subset \Z\subset\D^3$ with the kernel $K(\bz,\bw)$ on $\Z$ obtained by polarizing the identity $\norm{\si_1(\bw)}^2$ for $\bw\in\Z$, and $E_2$ is the line bundle associated to the \rkhs with the reproducing kernel obtained by polarizing $\norm{\si_2(\bw)}^2$. These observations lead us to consider the general case as presented in the previous section.

\subsection{Toy example II}
Let $\a,\b,\g,\d>0$ and $\H^{(\a,\b,\g,\d)}(\D^4)$ be the weighted Bergman module over $\D^4$ with the reproducing kernel 
$$K(\bz,\bw)=(1-z_1\ov{w}_1)^{-\a}(1-z_2\ov{w}_2)^{-\b}(1-z_3\ov{w}_3)^{-\g}(1-z_4\ov{w}_4)^{-\d}~\bz,\bw\in\D^4.$$ Let us take the submodule $\H_0^{(2)}$ consisting of functions in $\H^{(\a,\b,\g,\d)}(\D^4)$ vanishing of order $2$ along the diagonal set $\Delta=\{(z_1,z_2,z_3,z_4)\in\D^3:z_1=z_2=z_3=z_4\}$. Denote $\H_q^{(2)}$ as the quotient space $\H^{(\a,\b,\g,\d)}(\D^4)\ominus \H^{(2)}_0$. Then $\H^{(2)}_q$ is \rkhs with the reproducing kernel $$JK(z,w)=\begin{pmatrix}
         (1-z\ov{w})^2 & \b z(1-z\ov{w}) & \g z(1-z\ov{w}) & \d z(1-z\ov{w})\\
          \b \ov{w}(1-z\ov{w}) & \b(1+\b z\ov{w}) & \b\g z\ov{w} & \b\d z\ov{w}\\
      \g \ov{w}(1-z\ov{w})  &    \b\g z\ov{w}  & \g(1+\g z\ov{w}) & \g\d z\ov{w}\\
      \d \ov{w}(1-z\ov{w}) & \b\d z\ov{w} & \g\d z\ov{w} & \d(1+\d z\ov{w})
      \end{pmatrix}\times (1-z\ov{w})^{-\a-\b-\g-2}.$$ A similar computation as above yields that $XJ^{(2)}KX^*$ is a block diagonal matrix valued kernel where $X$ is the $4\times 4$ matrix 
$$X=\begin{pmatrix}
         1 & 0 & 0 & 0\\
          0 & 1 & 1 & 1\\
         0 &  -(\g+\d) & \b & \b\\
         0 & 0 & -\d & \g\end{pmatrix}.$$ So this $X$ corresponds a isometric bundle isomorphism $\Phi_X:J^{(2)}E|_{\Delta}\ra J^{(2)}E|_{\Delta}$ which acts on the global holomorphic frame $\{K(.,\ov{\bw}),\del_2 K(.,\ov{\bw}),\del_3 K(.,\ov{\bw}),\del_4 K(.,\ov{\bw})\}$ for the bundle $J^{(2)}E|_{\Delta}\ra\Delta$ as follows:
\Bea
\Phi_X(K(.,\ov{\bw})) &=& K(.,\ov{\bw})\\
\Phi_X(\del_2 K(.,\ov{\bw})) &=& (\del_2+\del_3+\del_4)K(.,\ov{\bw})\\
\Phi_X(\del_3 K(.,\ov{\bw})) &=& (-(\g+\d)\del_2+\b(\del_3+\del_4))K(.,\ov{\bw})\\
\Phi_X(\del_4 K(.,\ov{\bw})) &=& (-\d\del_3+\g\del_4)K(.,\ov{\bw}).
\Eea
Let $\Z_1=\{\bz\in\D^4:z_2=z_3=z_4\}$ and $\Z_2=\{\bz\in\D^4:z_3=z_4\}$. So we have that $\Delta\subset\Z_1\subset\Z_2\subset\D^4$. Set $\si_0(\bw)=K(.,\ov{\bw})$ and $\si_1(\bw)=(-\d\del_3+\g\del_4)K(.,\ov{\bw})$ for $\bw\in\Z_2$.   It then turns out that $\<\si_0(\bw),\si_1(\bw)\>=0$ for $\bw\in\Z_2$. Consequently, both $\del_2\si_0(\bw)$ and $(\del_3+\del_4)\si_0(\bw)$ are also perpendicular to $\si_1(\bw)$. Thus, it shows that $$J^{(2)}E|_{\Delta}=\hat{E}_0|_{\Delta}\oplus E_1|_{\Delta}$$ where $\hat{E}_0|_{\Delta}$ is the {\h} over $\Delta$ of rank $3$ with the global holomorphic frame $\{\si_0(\bw),\del_2\si_0(\bw),(\del_3+\del_4)\si_0(\bw)\}$ and $E_1|_{\Delta}$ is the {\hl} over $\Delta$ with global holomorphic frame $\{\si_1(\bw)\}$. Now we note that if we identify $\Z_2$ with $\D^3$ via the biholomorphism $\phi:\D^3\ra\Z_2$ defined by $\phi(z_1,z_2,z_3)=(z_1,z_2,z_3,z_3)$ then the diagonal subset of $\D^3$ gets mapped onto the set $\Delta$ and $\phi_*(\del_2)=\del_2$ and $\phi_*(\del_3)=\del_3+\del_4$. Therefore, the vector bundle $\hat{E}_0$ is nothing else but the jet bundle $J^{(2)}E_0|_{\Delta}$ of the {\h} $E_0\ra\Z_2$ relative to $\Delta$ where $E_0\ra\Z_2$ is the vector bundle associated to the adjoint of the tuple of multiplication operators on $\H^{(\a,\b,\g+\d)}(\D^3)=\H^{(\a,\b\g,\d)}(\D^4)|_{\Z_2}:=\{f|_{\Z_2}:f\in \H^{(\a,\b\g,\d)}(\D^4)\}$. Thus, rewriting the equation above we have that $$J^{(2)}E|_{\Delta}=J^{(2)}E_0|_{\Delta}\oplus E_1|_{\Delta}.$$ 

Now as above one can see that $K(.,\ov{\bw})$ and $(-(\g+\d)\del_2+\b(\del_3+\del_4))K(.,\ov{\bw})$ are perpendicular for $\bw\in\Z_1$. Consequently, $(\del_2+\del_3+\del_4)K(.,\ov{\bw})$ is also orthogonal to $(-(\g+\d)\del_2+\b(\del_3+\del_4))K(.,\ov{\bw})$. Thus, we have that $$J^{(2)}E_0|_{\Delta}=\hat{F}_0|_{\Delta}\oplus F_1|_{\Delta}$$ where $\hat{F}_0|_{\Delta}\ra\Delta$ is the vector bundle with global holomorphic frame $\{K(.,\ov{\bw}),(\del_2+\del_3+\del_4)K(.,\ov{\bw})\}$ and $\hat{F}_0|_{\Delta}$ is the line bundle with global holomorphic frame $\{(-(\g+\d)\del_2+\b(\del_3+\del_4))K(.,\ov{\bw})\}$. Again as above we observe that $\hat{F}_0|_{\Delta}\ra\Delta$ is the jet bundle $J^{(2)}F_0|_{\Delta}\ra\Delta$ of the vector bundle $F_0$ associated to the adjoint of the tuple of multiplication operators on $\H^{(\a,\b+\g+\d)}(\D^2):=\H^{(\a,\b,\g,\d)}(\D^4)|_{\Z_1}=\{f|_{\Z_1}:f\in \H^{(\a,\b,\g,\d)}(\D^4)\}$. Thus, we have that
$$J^{(2)}E|_{\Delta}=J^{(2)}F_0|_{\Delta}\oplus F_1|_{\Delta}\oplus E_1|_{\Delta}.$$

\bibliographystyle{plain}

%\bibliography{Motherbibliography}
\end{document}